\documentclass[a4paper,centertags,oneside,12pt]{amsart}
\usepackage{mathrsfs}
\usepackage{appendix}
\usepackage{amssymb}
\usepackage{fancyhdr}
\usepackage{charter}
\usepackage{typearea}
\usepackage{pdfsync}
\usepackage{color}
\usepackage[colorlinks,linkcolor=blue,anchorcolor=blue,citecolor=green]{hyperref}
\usepackage{mathrsfs}
\usepackage[a4paper,top=3cm,bottom=3cm,left=3cm,right=3cm]{geometry}

\usepackage{enumitem}
\usepackage{color}
\setlist[1]{itemsep=5pt}
\newcommand{\comment}[1]{}
\newcommand{\ov}[1]{\overline{#1}}

\makeatletter
      \def\@setcopyright{}
      \def\serieslogo@{}
      \makeatother

\newcommand{\mbb}{\mathbb}

\newcommand{\pa}{\partial}

\newcommand{\Om}{\Omega}

\newcommand{\be}{\beta}

\newcommand{\Ga}{\Gamma}

\renewcommand{\Im}{\operatorname{Im}}

\newcommand{\abs}[1]{\left\vert#1\right\vert}

\newtheorem{theorem}{Theorem}[section]
\newtheorem{lemma}[theorem]{Lemma}
\newtheorem{corollary}[theorem]{Corollary}

\newtheorem{proposition}[theorem]{Proposition}

\newtheorem{definition}[theorem]{Definition}

\newtheorem{remark}[theorem]{Remark}

\numberwithin{equation}{section}

\begin{document}
\title{The invariant Szeg\H{o} metric on Egg domains}
\author{Anjali Bhatnagar}
\address{School of Mathematics and Statistics, Wuhan University, Wuhan, Hubei 430072, China.}
\email{anj.bhatngr28@gmail.com}
\author{Jiliang Fan}
\address{School of Mathematics and Statistics, Wuhan University, Wuhan, Hubei 430072, China.}
\email{jiliangfan@whu.edu.cn}

\begin{abstract}

  \bigskip\bigskip

We study the Fefferman--Szeg\H{o} metric on egg domains
\[
\mathcal D_{2m}=\{(z,w)\in\mathbb C^2: |z|^2+|w|^{2m}<1\},\qquad\qquad\qquad m\in\mathbb Z^+.
\]
Our first main result establishes the existence of the Fefferman--Szeg\H{o} kernel on $\mathcal{D}_{2m}$ by verifying that the Fefferman weight lies in the Muckenhoupt class $A_2(\partial\mathcal{D}_{2m})$. We then derive an explicit closed-form expression for this kernel, demonstrate that its blowup occurs precisely on the boundary diagonal, and determine its boundary asymptotic behaviour. Using this kernel, we compute the associated Fefferman--Szeg\H{o} metric and its Ricci curvature. As applications, we prove several rigidity results: the metric is K\"ahler--Einstein if and only if $m=1$; proportionality to the Bergman metric or to some complete K\"ahler metric $g_m^{\mathcal D_{2m}}$ is also equivalent to $m=1$. Finally, we establish the vanishing of the $L^2$-cohomology outside the middle dimension for the Fefferman--Szeg\H{o} metric.
\end{abstract}

\maketitle

\section{Introduction}
The purpose of this article is to contribute to the theory of invariant Szeg\H{o} metrics. The Szeg\H{o} metric is defined using the Szeg\H{o} kernel, in analogy with the Bergman metric. Unlike the Bergman metric, the Szeg\H{o} metric is not biholomorphically invariant in general. This is because the Euclidean surface area measure $\sigma_{\operatorname E}$ does not transform well under biholomorphic mappings. 

To resolve this issue, Fefferman introduced an invariant surface area measure $\sigma_{\operatorname F}$ on $C^\infty$-smooth bounded strongly pseudoconvex domains \cite[p.~259]{f79}. The Szeg\H{o} metric associated with $\sigma_{\operatorname F}$ is called the Fefferman--Szeg\H{o} metric. Its systematic study was initiated by Barrett--Lee \cite{bl14}, who investigated, in particular, its relationship with the Bergman and Carath\'eodory metrics. This Fefferman--Szeg\H{o} metric has attracted considerable attention in recent years. Krantz studied representative coordinates, analytic continuation, and completeness properties of this metric in \cite{k19}; see also \cite{k21}. Boundary behaviour and related invariants of the Fefferman--Szeg\H{o} metric were studied in \cite{m08,b26}. In complex dimension one, this metric has been studied through its intrinsic properties, including geodesics (with higher-dimensional analogues treated in \cite{b25}), curvature, and $L^2$-cohomology in \cite{bb24}. Barrett \cite{b06} and Gupta \cite{g17,Jg17} employed ideas from convex geometry to extend Fefferman's surface area measure to general bounded pseudoconvex domains. A natural question raised by Yuan \cite{y25} is how these extended measures fit into the invariant Szeg\H{o} theory.

A classical problem in complex geometry is to characterize the unit ball among bounded domains using invariant metrics. Cheng's conjecture---that a bounded $C^\infty$-smooth strongly pseudoconvex domain in $\mathbb C^n$ with K\"{a}hler--Einstein Bergman metric must be biholomorphic to the unit ball $\mathbb B^n$---was resolved in dimension two \cite{fw97,ns06}, in higher dimensions \cite{hx21}, and extended to Stein spaces with compact $C^\infty$-smooth strongly pseudoconvex boundary \cite{hl23}. Addressing Yau's broader question \cite{y82}, Savale--Xiao \cite{sx25} confirmed it for finite-type domains in $\mathbb C^2$, with subsequent extensions to real analytic and $h$-extendible pseudoconvex domains in $\mathbb C^n$ by Hsiao, Huang and Li \cite{hhl26}. Moreover, a recent result shows that in $\mathbb C^n$, an unbounded pseudoconvex domain cannot carry a K\"{a}hler--Einstein Bergman metric if its boundary contains a non-smooth strongly pseudoconvex polyhedral point \cite{hjl25}.

 In a parallel development, Yuan \cite{y25} established an analogue of Cheng’s conjecture for the Fefferman--Szeg\H o metric. Further analytic and rigidity properties of the Fefferman–Szeg\H o metric---including $L^2$-cohomology, bounded geometry, local sphericity of the boundary, and a ball characterization result---were obtained in \cite{BF26}. In the same work \cite{y25}, following a natural definition of Duong--Lanzani--Li--Wick \cite{dlw25, dllw26} of weighted Hardy space, Yuan studied the weighted  Szeg\H o metric on $C^2$-smooth bounded domains $\Om$. In this setting one considers measures of the form
\[
d\sigma_\varpi=\varpi d\sigma_{\operatorname E},
\]
where $\varpi$ is a nonnegative locally integrable function on $\partial\Omega$. Under the assumption that $\varpi$ belongs to the Muckenhoupt class $A_2(\partial\Omega)$, the corresponding weighted Hardy space is a reproducing kernel Hilbert space. Consequently, the associated weighted Szeg\H{o} kernel and weighted Szeg\H{o} metric are well defined. This provides a natural framework for extending the Fefferman--Szeg\H{o} metric beyond the category of strongly pseudoconvex domains. 

In this paper, we develop the theory of the Fefferman--Szeg\H o metric $g_{\operatorname{FS}}^{\mathcal{D}_{2m}}$ on the family of egg domains \[\mathcal{D}_{2m}=\{|z|^2+|w|^{2m}<1\}\subset \mathbb C^2.\]
For $m=1$, $\mathcal D_{2m}$ is the unit ball $\mathbb B^2$ in $\mathbb C^2$; for $m>1$, the domain is weakly pseudoconvex and its boundary has Levi degeneracy along the circle
\[
\{(e^{i\theta},0):\theta\in\mathbb R\}.
\]
Thus, this family provides a natural model for studying the Fefferman--Szeg\H{o} theory in the presence of weak pseudoconvexity. Yuan illustrated this on  $\mathcal{D}_{4}$ in Example 2.20 of \cite{y25}, verifying that the weight function in $\sigma_{\operatorname{F}}$ lies in $A_2(\pa \mathcal{D}_4)$.
We extend this result to the entire family $\mathcal{D}_{2m}$ using geometric arguments, in contrast to the analytic approach employed there.
\begin{theorem}\label{feff_density_Ap}
For any $m\in\mathbb Z^+$, let $\mathcal{D}_{2m}=\{(z,w)\in\mathbb{C}^2:|z|^2+|w|^{2m}<1\}$. Let $\varpi_{\operatorname{F}}$ denote the weight function in the Fefferman surface area measure $\sigma_{\operatorname{F}}$ on $\pa \mathcal{D}_{2m}$, i.e., $d\sigma_{\operatorname{F}}=\varpi_{\operatorname{F}}d\sigma_{\operatorname{E}}$. Then, we have \[\varpi_{\operatorname{F}}\in A_2(\pa\mathcal{D}_{2m}).\] 
\end{theorem} 
This result ensures the existence of the Fefferman–Szeg\H o kernel on $\mathcal{D}_{2m}$. We then compute the Fefferman--Szeg\H o kernel $S_{\mathcal{D}_{2m}}$ in its closed form.
\begin{theorem}\label{feff-sgo-egg-doms}
    For any $m\in\mathbb Z^+$, let $\mathcal{D}_{2m}=\{(z,w)\in\mathbb{C}^2:|z|^2+|w|^{2m}<1\}$. Then the Fefferman-Szeg\H{o} kernel is given by
    \begin{equation}\label{diag_sgoker}
        S_{\mathcal{D}_{2m}}(\zeta,\eta)=c_{\mathcal{D}_{2m}}\frac{(m+2)(1-z_1 \ov z_2)^{\frac{1}{m}}+(1-m)w_1\ov w_2}{3m(1-z_1\ov z_2)^{\frac{4m-1}{3m}}\big((1-z_1 \ov z_2)^{\frac{1}{m}}-w_1\ov w_2\big)^2}
    \end{equation} 
    where $\zeta=(z_1,w_1),\;\eta=(z_2,w_2)\in \mathcal{D}_{2m}\;\;\text{ and }\;c_{\mathcal{D}_{2m}}=m^{\frac{1}{3}}/2\pi^2$.
\end{theorem}
A key consequence of this explicit formula is that the kernel $S_{\mathcal{D}_{2m}}(\zeta,\eta)$ blows up precisely on the boundary diagonal $E=\{(\zeta,\eta)\in\partial\mathcal{D}_{2m}\times\partial\mathcal{D}_{2m}:\zeta=\eta\}$; see Proposition \ref{deno-feff-sgo-ker}. In particular, for each fixed $\eta\in\mathcal{D}_{2m}$, the function $S_{\mathcal{D}_{2m}}(\cdot,\eta)$ extends holomorphically to a neighbourhood of the closure of $\mathcal{D}_{2m}$; see Corollary \ref{Sgoker_ext_holo}. Consequently, the space of holomorphic functions on the closure of $\mathcal{D}_{2m}$ is dense in the Fefferman--Hardy space for $\mathcal{D}_{2m}$; see Corollary \ref{density}. Hence the definition of the Fefferman–Hardy space on $\mathcal{D}_{2m}$ considered in this paper coincides with the one for $C^\infty$-smooth bounded strongly pseudoconvex domain studied in \cite{bl14}. The boundary asymptotics of the diagonal values of the kernel $S_{\mathcal{D}_{2m}}$ are derived in Proposition \ref{bdry-asym-sgo}, in analogy with \cite{a78}.

Furthermore, following Barrett--Lee, we consider the biholomorphic invariant $SK_{\mathcal{D}_{2m}}$ defined using the Bergman and Fefferman--Szeg\H o kernel. We study its limiting behaviour and establish its boundedness; in contrast to the strongly pseudoconvex case, this invariant is non-constant unless $m=1$; see Propositions \ref{SK-lim} and \ref{SK:bdd}.

Our next result provides a characterization of the unit ball $\mathbb{B}^2$ within the family of egg domains $\mathcal{D}_{2m}$ using the Einstein property and the comparability with the natural metrics such as the Bergman metric $g_{\operatorname{B}}^{\mathcal D_{2m}}$ and complete K\"ahler metric $g_m^{\mathcal D_{2m}}$. 
\begin{theorem}\label{main:egg}
    For $m\in\mathbb Z^+$, let \(\mathcal D_{2m}=\{(z,w)\in\mathbb C^2~:~|z|^2+|w|^{2m}<1\}\subset\mathbb C^2\). Then, the following statements are equivalent:
    \begin{enumerate}
        \item[(i)] the Fefferman--Szeg\H o metric \(g_{\operatorname{FS}}^{\mathcal D_{2m}}\) is K\"ahler--Einstein;
        \item[(ii)] \(g_{\operatorname{FS}}^{\mathcal D_{2m}}=\lambda g_{\operatorname{B}}^{\mathcal D_{2m}}\) for some \(\lambda>0\);
        \item[(iii)] \(g_{\operatorname{FS}}^{\mathcal D_{2m}}=\lambda g_m^{\mathcal D_{2m}}\) for some \(\lambda>0\);
        \item[(iv)] $m=1$.
    \end{enumerate}
\end{theorem}
The proof of part (i) of Theorem \ref{main:egg} hinges on the explicit expressions for the Fefferman–Szeg\H{o} metric and its associated Ricci curvature on $\mathcal D_{2m}$, which we state in the following theorem.
\begin{theorem}\label{FSmetric-Ricci}
    For any $m\in\mathbb Z^+$, let $\mathcal{D}_{2m}=\{(z,w)\in\mathbb{C}^2:|z|^2+|w|^{2m}<1\}$. For $w\in\mathbb D$, set
    \[
    r:=\frac{m-1}{m+2},\qquad t:=\frac{1-|w|^2}{1-r|w|^2}.
    \]
    Then, the Fefferman--Szeg\H o metric $g_{\operatorname{FS}}^{\mathcal D_{2m}}$ at \((0,w)\) is given by
    \[
    g_{\operatorname{FS}}^{\mathcal D_{2m}}(0,w)
    =
    \begin{pmatrix}
    \dfrac{\alpha(t)}{t(1+2r)} & 0\\[8pt]
    0 & \beta(t)\,\dfrac{(1-rt)^2}{t^2(1-r)^2}
    \end{pmatrix},
    \]
    where
    \[
    \alpha(t):=2+rt+rt^2,\qquad \beta(t):=2-rt^2.
    \]
    Moreover, the Ricci curvature at \((0,w)\) is diagonal and has the form
    \[
    \operatorname{Ric}_{\operatorname{FS}}^{\mathcal D_{2m}}(0,w)
    =
    \begin{pmatrix}
    -\dfrac{2A(t)+4\alpha(t)}{t\alpha(t)(1+2r)}
    +\dfrac{B(t)}{(2r+1)t(2-rt^2)}
    +\dfrac{C(t)}{t\alpha(t)(1+2r)}
    & 0\\[12pt]
    0 & D(t)+E(t)+F(t)
    \end{pmatrix},
    \]
    with the auxiliary functions
    \[
    \begin{aligned}
    A(t)&=-r^{2}t^{4}+(3r^{2}+2r)t^{3}+2r^{2}t^{2}-2,\\
    B(t)&=-2r^{2}t^{4}+(r+r^{2})t^{3}+2(1+r)t-4,\\
    C(t)&=(1-t)(1-rt)(2-rt^2),
    \end{aligned}
    \]
    and
    \[
    \begin{aligned}
    D(t)&=
    \frac{-2t^2(1-r)^2(2-r^2t^2)-8t(1-r)(1-t)(2-r^3t^3)-6(1-t)^2(2-r^4t^4)}
    {t^2(1-r)^2(2-rt^2)},\\[8pt]
    E(t)&=
    \frac{4(1-t)(1-rt)(2-r^2t^3)^2}
    {t^2(1-r)^2(2-rt^2)^2},\\[8pt]
    F(t)&=
    -\frac{(1-rt)^2(2-rt^2)}{(1-r)t\alpha(t)}
    -\frac{2(1-t)(1-rt)^2(2-r^2t^3)}
    {(1-r)^2t^2\alpha(t)}
    +\frac{(1-t)(1-rt)^3(2-rt^2)^2}
    {(1-r)^2t^2\alpha(t)^2}.
    \end{aligned}
    \]
\end{theorem}

\medskip

Finally, we study the $L^2$-cohomology of the Fefferman--Szegő metric.
The $L^2$-cohomology for the Bergman metric on $C^\infty$-smooth bounded strongly pseudoconvex domains in $\mathbb{C}^n$ was studied by Donnelly–Fefferman \cite{df83} and Donnelly \cite{d94}, and in a more general setting by McNeal \cite{m02} and Ohsawa \cite{o89}, among others.
Here, we establish the following vanishing theorem for the $L^2$-cohomology of the Fefferman--Szeg\H o metric on $\mathcal{D}_{2m}$.
\begin{theorem}\label{van-L2-coho}
For any $m \in\mathbb Z^+$, let $\mathcal{D}_{2m}=\left\{(z,w)\in \mathbb{C}^2: |z|^2 + |w|^{2m} < 1\right\}$. Let $\mathcal{H}_2^k(\mathcal{D}_{2m})$ denote the space of square-integrable harmonic $k$-forms on $\mathcal{D}_{2m}$ with respect to $g_{\operatorname{FS}}^{\mathcal{D}_{2m}}$. Then for all $k \neq 2$, we have
\[\mathcal{H}_2^k(\mathcal{D}_{2m}) = 0.
\]

\end{theorem}
We expect that the present techniques generalize to higher-dimensional egg domains, but we restrict our treatment to $\mathbb{C}^2$
  in order to keep the exposition as transparent as possible.
\section{Preliminaries}
\subsection*{Notation}
We use the following notations throughout the paper.

\begin{itemize}
\item Let $\mathbb{Z}^+$ denote the positive integers. The unit ball in $\mathbb{C}^n$ is 
$\mathbb{B}^n:=\{z\in\mathbb{C}^n: |z_1|^2+\cdots+|z_n|^2<1\}$; we write $\mathbb{B}^1=\mathbb{D}$.
\item For $i,j=1,\dots,n$, we use the shorthand notation $\partial_i:=\partial/\partial z_i$, $\partial_{\bar j}:=\partial/\partial \bar z_j$, $\partial_{i\bar j}:=\partial_i\partial_{\bar j}
=\partial^2/\partial z_i\,\partial\bar z_j$ and so on.

\item For any two nonnegative quantities $A$ and $B$, write $A\lesssim B$ if $A\le C B$ for some constant $C>0$ independent of the relevant parameters; write $A\approx B$ if both $A\lesssim B$ and $B\lesssim A$ hold.

\item For a bounded domain $\Omega\subset\mathbb C^n$ with $C^2$-smooth boundary $\partial\Omega$, let $\operatorname{dist}(z,\partial\Omega)$ denote the Euclidean distance from $z\in\Omega$ to the boundary $\partial\Omega$. Denote by $\mathcal{O}(\Omega)$ the space of holomorphic functions on $\Omega$, and by $J_{\mathbb{C}}F$ the complex Jacobian matrix of a biholomorphism $F$.

\item Let $\sigma_{\mathrm{E}}$ be the Euclidean surface measure on $\partial\Omega$. For $p>1$ and a nonnegative locally integrable weight $\varpi$ on $\partial\Omega$, define the weighted Lebesgue space $L_\varpi^p(\partial\Omega)$ by
\[
\|f\|_{L_\varpi^p}^p := \int_{\partial\Omega} |f|^p \varpi\, d\sigma_{\mathrm{E}} < \infty.
\]
When $\varpi\equiv 1$, we write $L^p(\partial\Omega)$.
\end{itemize}

\subsection*{Spaces of Homogeneous Type} We recall the spaces of homogeneous type introduced by Coifman and Weiss \cite{cw71,cw77}; this framework is required to define the Fefferman--Hardy space. Let $(X,\mu)$ be a measure space. A function $d:X\times X\to[0,\infty)$ is called a \emph{quasimetric} if for some $K>1$ and all $x,y,z\in X$,
\begin{enumerate}
\item[(i)] $d(x,y)=0$ if and only if $x=y$;
\item[(ii)] $d(x,y)=d(y,x)$;
\item[(iii)] $d(x,y)\le K\bigl(d(x,z)+d(z,y)\bigr)$.
\end{enumerate}
For $x\in X$ and $r>0$, write $\mathrm{B}_d(x,r):=\{y\in X:d(x,y)<r\}$; assume each such ball is $\mu$-measurable.

The triple $(X,d,\mu)$ is a \emph{space of homogeneous type} if, in addition, there exist constants $K_1,K_2>0$ such that, for all $x,x_1,x_2\in X$ and $r>0$,
\[
\operatorname{B}_d(x_1,r)\cap \operatorname{B}_d(x_2,r)\neq \emptyset
\quad\Longrightarrow\quad
\operatorname{B}_d(x_2,r)\subset \operatorname{B}_d(x_1,K_1r),
\]
and
\[
\mu\bigl(\operatorname{B}_d(x,K_1r)\bigr)
\le
K_2\,\mu\bigl(\operatorname{B}_d(x,r)\bigr).
\]

\begin{definition}
Let $(X,d,\mu)$ be a space of homogeneous type and  $p>1$. A nonnegative locally integrable function $\varpi$ on $X$ belongs to the Muckenhoupt class $A_p(X,d,\mu)$ if
\[
[\varpi]_{A_p(X,d,\mu)}
:=
\sup_{\operatorname{B}}
\left\langle \varpi \right\rangle_{\operatorname{B}}
\left\langle \varpi^{-\frac{1}{p-1}} \right\rangle_{\operatorname{B}}^{\,p-1}
<\infty,
\]
where the supremum is taken over all quasimetric balls $\operatorname{B}\subset X$ with $0<\mu(\operatorname{B})<\infty$, and
\[
\left\langle \phi \right\rangle_{\operatorname{B}}
:=
\frac{1}{\mu(\operatorname{B})}
\int_{\operatorname{B}}\phi\,d\mu.
\]
When $X=\partial\Omega$ with the quasimetric and measure fixed, we write $\varpi\in A_p(\partial\Omega)$.
\end{definition}

\subsection*{Weighted Hardy Spaces} The definition of weighted Hardy spaces is based on non--tangential maximal functions, following the classical approach of Stein \cite{s72}. In the setting of bounded strongly pseudoconvex domains with $C^2$-smooth boundary, Duong--Lanzani--Li--Wick developed a weighted theory for such spaces with respect to Muckenhoupt measures \cite{dlw25, dllw26}. Yuan subsequently used this framework to formulate weighted--Szeg\H{o} kernels on bounded $C^2$-smooth domains \cite{y25}.

Let $\Omega\subset\mathbb C^n$ be a bounded domain with $C^2$-smooth boundary. If $\varpi$ is a nonnegative locally integrable function on $\partial\Omega$, write
\[
d\sigma_\varpi:=\varpi\,d\sigma_{\operatorname E}.
\]
For $\xi\in\partial\Omega$, let $\nu_\xi$ be the outward unit normal vector at $\xi$. For $z\in\Omega$, let $\delta_\xi(z)$ denote the minimum of  $\mathrm{dist}(z,\pa\Om)$ and the Euclidean distance from $z$ to the tangent plane to $\partial\Omega$ at $\xi$. For $\alpha>0$, define
\[
\Gamma_\alpha(\xi)
:=
\left\{z\in\Omega:
|(\xi-z)\cdot\overline{\nu_\xi}|<(1+\alpha)\;\delta_\xi(z),
\quad
|\xi-z|^2<\alpha\,\delta_\xi(z)
\right\}.
\]
For $f\in\mathcal O(\Omega)$, its non-tangential maximal function is
\[
\mathcal N(f)(\xi)
:=
\sup_{z\in\Gamma_\alpha(\xi)}|f(z)|,
\qquad \xi\in\partial\Omega.
\]
If the limit of $f(z)$ exists as $z\to\xi$ within $\Gamma_\alpha(\xi)$ for all $\alpha>0$, we denote it by $f^*(\xi)$ and call it the non--tangential boundary value of $f$.

\begin{definition}{(Duong--Lanzani--Li--Wick).} The weighted Hardy space is defined by
\[
\operatorname{H}^2_{\varpi}(\partial\Omega) = \left\{ f \in \mathcal{O}(\Omega) : \mathcal{N}(f) \in L_{\varpi}^2(\partial\Omega) \right\}.
\]
\end{definition}
Next, we recall the following result from \cite[Proposition 2.4]{y25}, which is a reformulation of \cite[Proposition 1.4]{dlw25}.

\begin{proposition}{(Duong--Lanzani--Li--Wick).} \label{DLLW}
If $\varpi \in A_2(\partial\Omega)$, then there exists $p>1$ such that
\[
\|\mathcal{N}(f)\|_{L^p(\partial\Omega)} \lesssim \|\mathcal{N}(f)\|_{L_{\varpi}^2(\partial\Omega)}.
\]
Moreover, every $f \in \operatorname{H}^2_{\varpi}(\partial\Omega)$ has a non--tangential limit $f^*$ for almost every $\xi \in \partial\Omega$ with respect to $d\sigma_\varpi$, satisfying
\[
\|f^*\|_{L_{\varpi}^2(\partial\Omega)} \approx \|\mathcal{N}(f)\|_{L_{\varpi}^2(\partial\Omega)}.
\]
As a consequence, $\operatorname{H}_{\varpi}^2(\partial\Omega)$ is a Hilbert space.
\end{proposition}

 Proposition \ref{DLLW} allows us to equip $\operatorname{H}^2_{\varpi}(\partial\Omega)$ with the norm
\[
\|f\|_{\operatorname{H}_{\varpi}^2(\partial\Omega)}: = \left( \int_{\partial\Omega} |f^*|^2(\xi) \, d\sigma_\varpi(\xi) \right)^{\frac{1}{2}},
\]
 and the same proposition implies that $\operatorname{H}_{\varpi}^2(\pa\Om)\subset \operatorname{H}^p(\pa\Om)$ for some $p>1$, where $\operatorname{H}^p(\pa\Om)$ denotes the $p$-th Hardy space (see \cite{s72, ls16}). Thus, for any $f\in \operatorname{H}_{\varpi}^2(\pa\Om)$, the Poisson integral estimate yields \[
|f(z)|^p
\le
\int_{\partial\Omega} P(z,\xi)|f^*(\xi)|^p\,d\sigma_{\mathrm{E}}(\xi)
\lesssim
\frac{\|f^*\|_{L^p(\partial\Omega)}^p}{\operatorname{dist}(z,\partial\Omega)^{2n}}
\approx
\frac{\|\mathcal{N}(f)\|_{L^p(\partial\Omega)}^p}{\operatorname{dist}(z,\partial\Omega)^{2n}}
\lesssim
\frac{\|\mathcal{N}(f)\|_{L_{\varpi}^2(\partial\Omega)}^p}{\operatorname{dist}(z,\partial\Omega)^{2n}},
\]
where the first inequality follows from \cite[Section 10]{s72} or \cite[Proposition 1]{ls16}. Consequently, the evaluation functional $\mathcal{E}_z:\operatorname{H}^2_{\varpi}(\partial\Omega)\to\mathbb C$ is bounded. Hence $\operatorname{H}^2_{\varpi}(\partial\Omega)$ is a reproducing kernel Hilbert space (for different arguments, see \cite[P. 5]{y25}).
We summarize the preceding discussion in the following result.
\begin{proposition}[\cite{dlw25,y25}] \label{H2reproducing}
If $\varpi\in A_2(\partial\Omega)$, then $\operatorname{H}^2_{\varpi}(\partial\Omega)$ is a reproducing kernel Hilbert space. Its reproducing kernel $S_{\varpi}(z,w)$ satisfies
\[
f(z)=\int_{\partial\Omega} S_{\varpi}(z,\xi)\,f(\xi)\,d\sigma_\varpi(\xi)
\qquad\text{for all } f\in \operatorname{H}^2_{\varpi}(\partial\Omega),\ z\in\Omega.
\]

If $\{\phi_j\}_{j\ge 1}$ is a complete orthonormal basis of $\operatorname{H}^2_{\varpi}(\partial\Omega)$, then

\begin{equation}\label{series_exp}
    S_{\varpi}(z,w)=\sum_{j\ge 1}\phi_j(z)\overline{\phi_j(w)},
\end{equation}
with uniform convergence on compact subsets of $\Omega\times\Omega$.
\end{proposition}

\begin{remark}
If $\varpi$ is a continuous positive function on $\partial\Omega$, then $\varpi\in A_2(\partial\Omega)$. Consequently, the weighted Hardy space $\operatorname{H}^2_{\varpi}(\partial\Omega)$ coincides with the classical Hardy space $\operatorname{H}^2(\partial\Omega)$ as a set.
\end{remark}

\subsection*{Fefferman--Hardy Space}

Inspired by Fefferman's idea for $C^\infty$-smooth bounded strongly pseudoconvex domains \cite[P.~259]{fc74}, we define a nonnegative surface area measure on the boundary of a $C^2$-smooth bounded domain $\Omega=\{\rho<0\}\subset\mathbb C^n$ by
\[
d\sigma_{\operatorname{F}} := \varpi_{\operatorname{F}}\, d\sigma_{\mathrm{E}},
\qquad\qquad
\varpi_{\operatorname{F}} :=
\frac{1}{\abs{d\rho}}
\left|
\det\begin{pmatrix}
\rho & \partial_{\overline{j}}\rho \\[6pt]
\partial_i\rho_{i} & \partial_i \partial_{\overline{j}}\rho
\end{pmatrix}_{1\le i,j\le n}
\right|^{\frac{1}{n+1}}.
\]

\medskip
The measure $\sigma_{\operatorname{F}}$ is independent of the choice of defining function $\rho$. For each strongly pseudoconvex boundary point $\xi$, $\varpi_{\operatorname{F}}(\xi)>0$. If $\varpi_{\operatorname{F}}\in A_2(\partial\Omega)$, then by Proposition \ref{H2reproducing} the space $\operatorname{H}^2_{\varpi_{\operatorname{F}}}(\partial\Omega)$ is a reproducing kernel Hilbert space; we denote it by $\operatorname{H}^2_{\operatorname{F}}(\partial\Omega)$ and call it the \emph{Fefferman--Hardy space}. Its reproducing kernel $S(z,w)$ is called the \emph{Fefferman--Szeg\H o kernel}.

To state the transformation law for the Fefferman--Szegő kernel, we introduce the following notion.

\begin{definition}\label{biholo_cond}
    A biholomorphism $F:\Omega_1\to\Omega_2$ between two such domains is called a \textit{(T)-biholomorphism} if  
\begin{itemize}
    \item[(i)] $F$ extends to a $C^2$-smooth diffeomorphism $\overline{\Omega}_1\to\overline{\Omega}_2$;  
    \item[(ii)] the function $(\det J_{\mathbb C}F)^{\frac{n}{n+1}}$ admits a globally defined holomorphic branch on $\Omega_1$.
\end{itemize}
\end{definition}
\begin{remark}
If each $\Omega_i\subset\mathbb C^n$ is a $C^\infty$-smooth bounded strongly pseudoconvex domain, then Fefferman's celebrated theorem \cite{f79} implies that (i) always holds. This was later extended by Bell and Ligocka \cite{BL80,Bell81}. If $\Omega_1$ is simply connected, then (ii) holds. Since the egg domains $\mathcal{D}_{2m}$ considered in this paper are star-shaped, and hence simply-connected, (ii) is automatically satisfied and need not be addressed separately.
\end{remark}
We next recall the transformation formula for Fefferman surface area measure for which we need to mentioned the underlying domain $\Om$, we write $\varpi_{\operatorname{F}}^{\pa\Om},\sigma_{\operatorname{F}}^{\partial\Omega}$, $S_\Om$ and so on.
\begin{proposition}\cite{bl14,y25}\label{FS-transformation}
Let $F:\Om_1\to\Om_2$ be a biholomorphism that satisfy Defintion \ref{biholo_cond}, then
\[F^*(d\sigma_{\operatorname{F}}^{\pa\Om_2})=|\det J_{\mathbb C}F|^{\frac{2n}{n+1}}d\sigma_{\operatorname{F}}^{\pa\Om_1}.\]
 Consequently, we have
     \[S_{\Om_1}(z,w)= S_{\Om_2}(F(z),F(w))\det J_{\mathbb C}F(z)^{\frac{n}{n+1}}\overline{\det J_{\mathbb C}F(w)}^{\frac{n}{n+1}}.
    \]
\end{proposition}
    
\subsection*{Fefferman--Szeg\H{o} metric}
We now define the  Fefferman–Szeg\H{o} metric on $\Omega$. For any $z\in\Omega$, let $S_{\Omega}(z):=S_{\Omega}(z,z)$ denote the diagonal value of the Fefferman--Szeg\H{o} kernel. By \eqref{series_exp}, the function 
$\log S_{\Omega}(z)$ is a $C^\infty$-smooth, strictly plurisubharmonic function on $\Omega$---which defines a K\"ahler metric on $\Omega$ called the \textit{Fefferman–Szeg\H{o} metric} $g_{\operatorname{FS}}^{\Omega}$ as follows:
\[g_{\operatorname{FS}}^\Omega=\sum_{i,j=1}^ng_{i\bar j}(z)dz_id\overline{z}_j,\qquad\qquad g_{i\bar j}(z):=\partial_i \partial_{\overline{j}}\log S_{\Omega}(z).\]
 
By Proposition \ref{FS-transformation}, this metric is invariant under (T)-biholomorphisms, i.e., for all $z\in\Om,~~X\in\mbb C^n$,  \[g^{\Om_1}_{\operatorname{FS},z}\big(X,X\big)=g^{\Om_2}_{\operatorname{FS},F(z)}\big(J_{\mbb C}F(z)X,J_{\mbb C}F(z)X\big).\] Using the similar reasoning as for the Bergman metric, it can be shown that the Fefferman--Szeg\H o metric dominates the Carath\'eodory; see \cite{bl14,y25}. This implies that the Fefferman–Szeg\H{o} metric is complete on any $C^{\infty}$-smooth bounded strongly pseudoconvex domain, as well as on the family of egg domains $\mathcal{D}_{2m}$.

The associated metric matrix is then written as $G_{\Omega}:=\left[g_{i\bar j}(z)\right]_{i,j=1}^{n}$.
The Riemann curvature tensor associated with the Fefferman–Szeg\H{o} metric $g_{\operatorname{FS}}^{\Omega}$ is given by
\begin{equation}
    R_{i\bar j k\bar l}(z):=-g_{i\bar jk\bar l}(z)+\sum_{ \alpha,\beta=1}^{n} g_{ik\bar\alpha}(z)g^{\bar\alpha\beta}(z)g_{\beta\bar j\bar l}(z)\quad\text{for }\quad i,j,k,l\in\left\{1,\cdots, n\right\},
\end{equation}
where $(g^{i\bar j})$ denotes the inverse of the metric matrix $(g_{i\bar j})$, and the higher-order derivatives are defined as
\begin{equation*}
    g_{i\bar jk\bar l}:=\partial_i \partial_{\overline{j}} g_{k\bar l},\quad g_{ik\bar \alpha}:=\partial_k g_{i\bar\alpha},\quad g_{\beta \bar j\bar l}:=\partial_{\overline{j}} g_{\beta\bar l}.
\end{equation*}
As $g_{\operatorname{FS}}^{\Omega}$ is a K\"ahler metric, its curvature tensor satisfies the following standard symmetry relations:
$$R_{i\bar jk\bar l}=R_{k\bar li\bar j}=\overline{R_{j\bar il\bar k}}.$$

The Ricci curvature of $g_{\operatorname{FS}}^{\Omega}$, denoted by $\operatorname{Ric}_{\operatorname{FS}}^{\Omega}$, is given by 
\begin{equation*}
\operatorname{Ric}_{\operatorname{FS}}^{\Omega}=\sum_{i,j=1}^n \operatorname{Ric}_{i\bar j} dz_i d\overline{z}_j,\qquad\qquad \operatorname{Ric}_{i\bar j}:=-\partial_i\partial_{\bar j}\log \det G_{\Omega}(z).
\end{equation*}
Furthermore, the components of the Ricci curvature and the Riemann curvature tensor are related by the following standard identity:
\begin{equation}\label{Ric-ibarj}
         \operatorname{Ric}_{i\bar j} = \sum_{k,\ell=1}^n g^{\ell\bar k} R_{i\bar j k\bar \ell}, \qquad
    R_{i\bar j k\bar \ell}=-g_{i\bar j k\bar \ell}
    +\sum_{\alpha,\beta=1}^n g^{\alpha\bar\beta} g_{i\bar\beta k} g_{\alpha\bar j \bar \ell}.
    \end{equation}

The Fefferman-Szeg\H{o} metric $g_{\operatorname{FS}}^{\Omega}$ is called a K\"ahler-Einstein if there exists a real constant $\lambda$ such that on $\Omega$, 
\begin{equation*}
\operatorname{Ric}_{\operatorname{FS}}^{\Omega}=\lambda g_{\operatorname{FS}}^{\Omega}.
\end{equation*}
\section{Fefferman--Szeg\H o kernel} 
Our first goal is to establish the existence of the Fefferman--Szeg\H{o} kernel on the egg domain $\mathcal D_{2m}:=\bigl\{(z,w)\in\mathbb C^2 : \rho(z,w)=|z|^2+|w|^{2m}-1<0\bigr\}.$
In view of Proposition \ref{H2reproducing}, we first recall the notion of a space of homogeneous type on the boundary $(\partial\mathcal D_{2m},d,dS)$, as introduced in \cite{h99}.
  The quasimetric  
\[
d(\zeta,\eta)=|v(\zeta,\eta)|+|v(\eta,\zeta)|,\qquad
v(\zeta,\eta)=\langle\partial\rho(\zeta),\zeta-\eta\rangle,
\]
 where $\langle\cdot,\cdot\rangle$ denotes the standard complex bilinear pairing on $\mathbb C^2$. 
 
 For $\zeta=(z,w),\eta=(z_0,w_0)\in\mathbb C^2$,
\begin{equation}\label{defn:v}
    v(\zeta,\eta)=\bar z(z-z_0)+m\bar w|w|^{2m-2}(w-w_0).
\end{equation}
The associated surface area measure is  
\[
dS=(2\pi i)^{-2}\partial\rho\wedge\bar\partial\partial\rho.
\]
From \cite{h99}, one can see that $dS$ satisfies 
    \[dS\approx 1+m^2|w|^{2(m-1)} d\sigma_{\operatorname{E}}\approx d\sigma_{\operatorname{E}}.\]
\subsection{Existence of the Fefferman--Szego kernel}   In this subsection, we prove Theorem \ref{feff_density_Ap}. The proof proceeds in three steps. First, we compute the explicit asymptotic behaviour of the Fefferman weight function $\varpi_{\mathrm{F}}$ on $\partial \mathcal{D}_{2m}$, showing that it vanishes like $|w|^{2(m-1)/3}$ near the degenerate circle $L=\{(z,0):|z|=1\}$. Second, we establish a comparison lemma relating the quasimetric distance to the singular set $L$ with $|w|^{2m}$. Finally, we verify the $A_2$-condition by analyzing three cases for a quasimetric ball $\mathrm{B}=\mathrm{B}_d(\zeta,r)$: (i) far from $L$; (ii) radius bounded below; (iii) close to $L$ with small radius. The rotational invariance reduces the last case to a local computation near $\zeta_0=(1,0)\in L$.

\begin{proof}[Proof of Theorem \ref{feff_density_Ap}]
Recall that the Fefferman surface measure on $\partial\mathcal D_{2m}$ is  
\[
d\sigma_{\operatorname{F}}= \varpi_{\operatorname{F}}\,d\sigma_{\mathrm E},\qquad 
\varpi_{\operatorname{F}}= \frac{\bigl(\det(\rho_{i\bar j})\bigr)^{\frac1{3}}}{\|\nabla\rho\|},
\]
where $\rho(z,w)=|z|^2+|w|^{2m}-1$.

\subsection*{The asymptotic behaviour of $\varpi_{\operatorname{F}}$}
A direct computation gives  
\[
\rho_z=\bar z,\quad \rho_{\bar z}=z,\quad \rho_{z\bar z}=1,\quad
\rho_w=m\bar w|w|^{2m-2},\quad \rho_{\bar w}=mw|w|^{2m-2},\quad \rho_{w\bar w}=m^2|w|^{2m-2},
\]
and $\rho_{z\bar w}=\rho_{w\bar z}=0$. Hence  \[
\det(\rho_{i\bar j})=m^2|w|^{2m-2}.
\]

On the boundary $\{\rho=0\}$, we have $|z|^2=1-|w|^{2m}$, so  

\[
\|\nabla\rho\|=\sqrt{|z|^2+m^2|w|^{4m-2}}=
\sqrt{1-|w|^{2m}+m^2|w|^{4m-2}}\approx 1.
\]
 Therefore  
\begin{equation}\label{asym_weightfn}
    \varpi_{\mathrm F}(z,w)\approx |w|^{\frac{2m-2}{3}}. 
\end{equation}
The weight function $\varpi_{\operatorname{F}}$ vanishes on the circle  
\[
L:=\{(z,0)\in\mathbb C^2:|z|=1\}.
\]

Since $\varpi_{\operatorname{F}}$ vanishes on $L$, by rotational invariance of $d$ and $dS$ we consider the parametrisation near $\zeta_0=(1,0)\in L$ as follows
\[
\Phi:(-\pi,\pi]\times\mathbb D\to\partial\mathcal D_{2m},\qquad
\Phi(\theta,w)=\bigl(e^{i\theta}\sqrt{1-|w|^{2m}},\,w\bigr),
\]
where $\Im\Phi=\{(z,w)\in\pa\mathcal{D}_{2m}:|w|<1\}$.

By direct computation, it can seen that 
\begin{equation*}
    \Phi^*(dS)=\frac{m^2}{4\pi^2}|w|^{2m-2}d\theta\wedge dw\wedge d\overline{w}.
\end{equation*}
In polar coordinates $w=se^{i\varphi}$, this becomes
\begin{equation}\label{dS_polar-coordinates}
    \Phi^*(dS)=\frac{m^2}{2\pi^2}s^{2m-1}d\theta\; ds\; d\varphi\qquad\text{ and }\qquad \Phi^*(\varpi_{\operatorname{F}})\approx s^{\frac{2m-2}{3}}.
\end{equation}

\begin{lemma}
For small $|\theta|,|w|$, we have
\[
d(\Phi(\theta,w),\zeta_0) \approx |\theta| + |w|^{2m}.
\]
\end{lemma}
\begin{proof}
From \eqref{defn:v}, for $\zeta=(z,w)$ with $|w|<1$ we have
\[
v(\zeta,\zeta_0)=\overline{z}(z-1)+m|w|^{2m}.
\]
Write $z=re^{i\theta}$ with $r=\sqrt{1-t}$ where $t=|w|^{2m}$. Then
\[
\operatorname{Re}v(\zeta,\zeta_0) = 1+(m-1)t - r\cos\theta,\qquad 
\operatorname{Im}v(\zeta,\zeta_0) = -r\sin\theta.
\]
Using $r=1-t/2+O(t^2)$ and $\cos\theta=1-\theta^2/2+O(\theta^4)$, we obtain
\[
\operatorname{Re}v(\zeta,\zeta_0) = \frac{\theta^2}{2}\left(1-\frac{t}{2}\right) + \left(m-\frac12\right)t + O(\theta^4+t^2).
\]
For small $|\theta|,t$, the positive terms dominate, giving $|\operatorname{Re}v(\zeta,\zeta_0)|\gtrsim \theta^2+t$. Also $|\operatorname{Im}v(\zeta,\zeta_0)|=r|\sin\theta|\gtrsim |\theta|$. Hence $|v(\zeta,\zeta_0)|\gtrsim |\theta|+t$. The reverse inequality is straightforward, so $|v(\zeta,\zeta_0)|\approx |\theta|+t$. Because $|v(\zeta_0,\zeta)|\approx |v(\zeta,\zeta_0)|$, we have $d(\zeta,\zeta_0)\approx |\theta|+t$. Since $\zeta\in\Im\Phi$, we are done. 
\end{proof}
It follows that \text{for small enough } $\epsilon_0 > 0 \text{ there exist constants } c_1, c_2 > 0 \text{ such that}$
\begin{equation}\label{subsets_inclusion}
    \{|\theta| < c_1 \epsilon_0,\ |w| < c_1 \epsilon_0^{1/(2m)}\} \subset \Phi^{-1}\bigl(\operatorname{B}_d(\zeta_0,\epsilon_0)\bigr) \subset \{|\theta| < c_2 \epsilon_0,\ |w| < c_2 \epsilon_0^{1/(2m)}\}.
\end{equation}

\subsection*{A comparison lemma} We next need the following lemma which compares the quasimetric distance to the singular set $L$ with $|w|^{2m}$.
\begin{lemma}\label{lem:distL}
There exist constants $c,C>0$ such that for every $\zeta=(z,w)\in\partial\mathcal D_{2m}$,
\[
c|w|^{2m}\le d(\zeta,L)\le C|w|^{2m}.
\]
\end{lemma}
\begin{proof}
    \textit{Upper bound.} If $z\neq0$, take $\eta=(z/|z|,0)\in L$. Then 
    \[|v(\zeta,\eta)|\le (1+m)|w|^{2m}\qquad\text{ and }\qquad|v(\eta,\zeta)|=|w|^{2m},\] so we have \[d(\zeta,\eta)\le (m+2)|w|^{2m}.\] 
    If $z=0$ (so $|w|=1$), choose $\eta=(w,0)\in L$; then $v(\zeta,\eta)=m$. By \cite[Lemma 1]{h99}, we have $|v(\eta,\zeta)|\approx m$, hence $d(\zeta,\eta)\le C'm$ for some $C'$, and since $|w|^{2m}=1$, we obtain \[d(\zeta,L)\le C|w|^{2m}\qquad\text{ with }\qquad C=\max(m+2,\,C'm).\] The case $|z|=1\implies w=0$ gives $\zeta\in L$ and the upper bound holds trivially.

\medskip

\textit{Lower bound.} For any $\eta=(e^{i\theta},0)\in L$, set $\Delta=z-e^{i\theta}$. Then
\[
|v(\zeta,\eta)|=|\bar z\Delta+m|w|^{2m}|\ge \bigl| m|w|^{2m}-|\bar z\Delta|\bigr|
\ge \bigl| m|w|^{2m}-|\Delta|\bigr|,
\]
because $|\bar z\Delta|\le|\Delta|$. Hence
\[
d(\zeta,\eta)=|v(\zeta,\eta)|+|v(\eta,\zeta)|
\ge \bigl| m|w|^{2m}-|\Delta|\bigr|+|\Delta|
\ge m|w|^{2m}.
\]
 Taking the infimum over $\eta\in L$ yields $d(\zeta,L)\ge m|w|^{2m}$.
\end{proof}

\subsection*{Verification of the $A_2$ condition}

We finally show that for every quasimetric ball $\operatorname{B}=\operatorname{B}_d(\zeta,r)\subset\partial\mathcal D_{2m}$,
\[
\left\langle \varpi_{\operatorname{F}} \right\rangle_{\operatorname{B}}
\left\langle \varpi_{\operatorname{F}}^{-1} \right\rangle_{\operatorname{B}}
\]
is bounded above by a constant independent of $\operatorname{B}$.  

The proof splits into three cases, depending on $d_0:=d(\zeta,L)$
and $r$. Let $K$ be the quasimetric constant and fix $\delta>0$ sufficiently small such that $K(1+2K)\delta<\epsilon_0$ that will be useful further.

\subsection*{Case I} Let $d_0\ge 2Kr$.
 For any $\xi\in \operatorname{B}$, the quasitriangle inequality give  
\[
d(\xi,L) <K(r+d_0)\le \Bigl(K+\frac12\Bigr)d_0,
\]
and
\[\frac{d_0}{2K}\leq \frac{d_0-Kr}{K}<d(\xi,\eta)\qquad\qquad\text{ for any }\quad\eta\in L,\]
taking the infimium over $\eta\in L$ yields $d(\xi,L)\ge \dfrac{d_0}{2K}$. 

Thus on $\operatorname{B}$, by Lemma~\ref{lem:distL}, \[d_0\approx |w|^{2m}.\] 
Using (\ref{asym_weightfn}), this implies on $\operatorname{B}$,
\[
\varpi_{\mathrm F}(\xi)\approx d_0^{\frac{2m-2}{6m}},\qquad\qquad
\varpi_{\mathrm F}^{-1}(\xi)\approx d_0^{-\frac{2m-2}{6m}}.
\]
 Consequently  
\[
\frac1{|\operatorname{B}|_{dS}}\int_{\operatorname{B}}\varpi_{\mathrm F}\,d S\approx d_0^{\frac{2m-2}{6m}},\qquad
\frac1{|\operatorname{B}|_{dS}}\int_{\operatorname{B}}\varpi_{\mathrm F}^{-1}\,d S\approx d_0^{-\frac{2m-2}{6m}},
\]
and therefore $
\left\langle \varpi_{\mathrm{F}} \right\rangle_{\operatorname{B}}
\left\langle \varpi^{-1}_{\mathrm{F}} \right\rangle_{\operatorname{B}}
\approx 1$ uniformly for all such balls.

\medskip

\subsection*{Case II} Let $r\ge\delta$. Then  $|\operatorname{B}_d(\zeta,r)|_{dS}\ge |\operatorname{B}_d(\zeta,\delta)|\approx \delta^2>0$ by Lemma 1 and Lemma 3 of \cite{h99}.  
We first check the integrability of $\varpi_{\operatorname{F}}$ and $\varpi_{\operatorname{F}}^{-1}$ on $\partial\mathcal{D}_{2m}$. 

Using \eqref{dS_polar-coordinates}, we have
\[I_1:=\int_{\partial\mathcal{D}_{2m}}\varpi_{\mathrm F}\,dS\approx 2m^2 \int_{s=0}^{1} s^{\frac{8m-5}{3}} \, ds<\infty\]
because $(8m-5)/3>-1$ for all $m\geq 1$. 
   
\[ I_2:=\int_{\partial\mathcal D_{2m}}\varpi_{\mathrm F}^{-1}\,dS \approx 2m^2 \int_{s=0}^{1} s^{2m-1-\frac{2m-2}{3}}\, ds.\]
The above integral converges if and only if \[2m-1-\frac{2m-2}{3}>-1.\]

Thus, we have $\left\langle \varpi_{\mathrm{F}} \right\rangle_{\operatorname{B}}
\left\langle \varpi^{-1}_{\mathrm{F}} \right\rangle_{\operatorname{B}}\lesssim I_1 \cdot I_2<\infty.$

\medskip

\subsection*{Case III} Let $d_0<2Kr$ where $r<\delta$ is sufficiently small.  
There exists a point $\zeta_1\in L$ such that \[d(\zeta,\zeta_1)<2Kr.\] Define $R:=K(1+2K)r<\delta$,  $\operatorname{B}^*:=\operatorname{B}_d(\zeta_1,R)\text{ and }\operatorname{B}_0^*:=\operatorname{B}_d(\zeta_0,R)$. Using the quasitriangle inequality, we have \begin{equation}\label{ball_inclusion}
    \operatorname{B}\subset \operatorname{B}^*
\end{equation} By the doubling property of $(\partial\mathcal D_{2m},d,dS)$,  \begin{equation}\label{doubling}
    |\operatorname{B}^*|_{dS}\le C'|\operatorname{B}|_{dS}.
\end{equation}  
for some constant $C'>1$. By rotational invariance of $(\partial\mathcal D_{2m},d,dS)$, and using \eqref{ball_inclusion} and \eqref{doubling}, 
\begin{align*}
\langle\varpi_{\operatorname{F}}\rangle_{\operatorname{B}}&\lesssim \langle\varpi_{\operatorname{F}}\rangle_{\operatorname{B}^*}=\langle\varpi_{\operatorname{F}}\rangle_{\operatorname{B}^*_0};\\
\left\langle\varpi_{\operatorname{F}}^{-1}\right\rangle_{\operatorname{B}}&\lesssim \left\langle\varpi_{\operatorname{F}}^{-1}\right\rangle_{\operatorname{B}^*}=\left\langle\varpi_{\operatorname{F}}^{-1}\right\rangle_{\operatorname{B}^*_0}.
\end{align*}
Thus, it is suffices to focus on $\operatorname{B}_0^*$. So, consider
 \begin{multline*}
\langle\varpi_{\operatorname{F}}\rangle_{\operatorname{B}^*_0}\approx \frac{1}{R^2}\int_{\operatorname{B}^*_0}\varpi_{\operatorname{F}} \;dS\approx \frac{m^2}{2\pi^2}\frac{1}{R^2}2\pi\int_{\theta=-R}^R\int_{s=0}^{R^{\frac{1}{2m}}}s^{\frac{8m-5}{3}}\;ds\;d\theta\\
\approx \frac{1}{R}\left[s^{\frac{8m-5}{3}+1}\right]_{s=0}^{R^{\frac{1}{2m}}}=\frac{1}{R}\left[R^{\frac{8m-2}{6m}}\right]
     =R^{\frac{m-1}{3m}}.
 \end{multline*}
And,
 \begin{multline*}
    \left\langle\varpi_{\operatorname{F}}^{-1}\right\rangle_{\operatorname{B}^*_0}\approx \frac{1}{R^2}\int_{\operatorname{B}_0^*}\varpi_F^{-1}\; dS\\\approx\frac{m^2}{2\pi^2}\frac{1}{R^2}2\pi\int_{\theta=-R}^R\int_{s=0}^{R^{\frac{1}{2m}}}s^{2m-1-\frac{2m-2}{3}}\, ds\;d\theta
     \approx\frac{1}{R}\int_{s=0}^{R^{\frac{1}{2m}}}s^{2m-1-\frac{2m-2}{3}}ds. 
 \end{multline*}
 We have
\[\left\langle\varpi_{\operatorname{F}}^{-1}\right\rangle_{\operatorname{B}^*_0}\approx\frac{1}{R}R^{\frac{1-(m-1)}{3m}}=R^{\frac{-(m-1)}{3m}}.\]
 
Hence,
 \[ \langle\varpi_{\operatorname{F}}\rangle_{\operatorname{B}}\left\langle\varpi_{\operatorname{F}}^{-1}\right\rangle_{\operatorname{B}}\approx R^{\frac{m-1}{3m}}R^{-\frac{m-1}{3m}}=1.\] 
  This completes the proof. 
  \end{proof}
Next, we will derive a closed formula of the Fefferman Szeg\H o kernel $S_{\mathcal{D}_{2m}}$. 

\subsection{Closed form of the Fefferman--Szeg\H o kernel on $\mathcal{D}_{2m}$} We will need the following lemmas to derive the closed form of $S_{\mathcal{D}_{2m}}$.
 \begin{lemma}\label{tau}
     For $m\in\mathbb{Z}^+$, let $\mathcal{D}_{2m}=\{(z,w)\in\mathbb{C}^2:|z|^2+|w|^{2m}<1\}$. For any $\zeta=(z_1,w_1)$, $\eta=(z_2,w_2)\in \mathcal{D}_{2m}$, we have
     \begin{equation}\label{equat1}
         |w_1\overline w_2|<|1-z_1\overline z_2|^{\frac{1}{m}}.
     \end{equation}
 \end{lemma}
\begin{proof}
  For $\zeta=(z_1,w_1),\eta=(z_2,w_2)\in\mathcal{D}_{2m}$ we have $|w_i|^{2m}<1-|z_i|^2$ for $i=1,2$. Hence
\begin{equation}\label{modulus:w1w2}
    |w_1\overline w_2| < (1-|z_1|^2)^{\frac{1}{2m}}(1-|z_2|^2)^{\frac{1}{2m}}.
\end{equation}
A direct computation shows
\[
|1-z_1\overline z_2|^2 - (1-|z_1|^2)(1-|z_2|^2)=|z_1-z_2|^2\ge 0,
\]
so
\[
(1-|z_1|^2)^{\frac12}(1-|z_2|^2)^{\frac12}\le |1-z_1\overline z_2|.
\]
Raising both sides to the power $1/m$ yields
\begin{equation}\label{1/m-th}
    (1-|z_1|^2)^{\frac{1}{2m}}(1-|z_2|^2)^{\frac{1}{2m}}\le |1-z_1\overline z_2|^{\frac1m}. 
\end{equation}
Combining \eqref{modulus:w1w2} and \eqref{1/m-th} gives the desired inequality.
\end{proof}
\begin{lemma}
Let $U = (\overline{\mathcal{D}}_{2m}\times\overline{\mathcal{D}}_{2m}) \setminus E$, where $E = \bigl\{(\zeta,\eta)\in \partial\mathcal{D}_{2m}\times\partial\mathcal{D}_{2m} : \zeta = \eta \bigr\}$
is the boundary diagonal. For any $\zeta=(z_1,w_1),\eta=(z_2,w_2)\in U$, define $G(\zeta,\eta)=1-z_1\overline{z}_2$. Then $\operatorname{Re}\bigl(G(\zeta,\eta)\bigr) > 0$. 
\end{lemma}

\begin{proof}
Write $z_1\overline{z}_2 = |z_1||z_2| e^{i\theta}$ with $\theta\in[-\pi,\pi)$. Then
\[
\operatorname{Re}\bigl(G(\zeta,\eta)\bigr) = 1 - |z_1||z_2|\cos\theta.
\]

Since $\zeta,\eta\in\overline{\mathcal{D}}_{2m}$, we have $|z_1|\le 1$ and $|z_2|\le 1$. We will consider two cases.

\textbf{Case I:} $|z_1||z_2| < 1$. Then
\[
\operatorname{Re}\bigl(G(\zeta,\eta)\bigr) \ge 1 - |z_1||z_2| > 0.
\]

\medskip

\textbf{Case II:} $|z_1||z_2| = 1$. Because $|z_1|,|z_2|\le 1$, this forces $|z_1| = |z_2| = 1$. The definition of $\overline{\mathcal{D}}_{2m}$ then implies $|w_1|^{2m} \le 1-|z_1|^2 = 0$, hence $w_1 = 0$; similarly $w_2 = 0$. Thus $\zeta = (z_1,0)$ and $\eta = (z_2,0)$ lie on the boundary $\partial\mathcal{D}_{2m}$. The hypothesis $(\zeta,\eta)\notin E$ means $\zeta \neq \eta$, so $z_1 \neq z_2$. Consequently $e^{i\theta} = z_1\overline{z}_2 \neq 1$, and therefore $\cos\theta < 1$. It follows that
\[
\operatorname{Re}\bigl(G(\zeta,\eta)\bigr) = 1 - \cos\theta > 0.
\]

In both cases we obtain $\operatorname{Re}\bigl(G(\zeta,\eta)\bigr) > 0$, which completes the proof.
\end{proof}

\begin{corollary}\label{realana_(1-z_1barz_2)^alpha}
     For any $\alpha\in\mathbb{R}$, the function $(1-z_1\overline z_2)^\alpha$ is real-analytic on $U=(\overline{\mathcal D}_{2m}\times\overline{\mathcal D}_{2m})\setminus E$ where $E=\{(\zeta,\eta)\in\partial\mathcal{D}_{2m}\times\partial\mathcal{D}_{2m}:\zeta=\eta\}$ denotes the boundary diagonal.
\end{corollary}
We finally prove Theorem \ref{feff-sgo-egg-doms}.

\begin{proof}[Proof of Theorem \ref{feff-sgo-egg-doms}]
    Up to a set of measure zero $\{(z,0):|z|=1\}$, the boundary $\partial\mathcal{D}_{2m}$ is $C^\infty$-smoothly parametrised by \[z=e^{i\theta}\sqrt{1-s^{2m}},\qquad\qquad w=se^{i\phi},\] with $s=|w|\in (0,1)$ and $\theta,\phi\in [-\pi,\pi)$.  
    
    The Euclidean surface measure is
    \[d\sigma_{\operatorname{E}}=s\sqrt{1-s^{2m}+m^2s^{4m-2}}\;d\theta\; ds\; d\varphi.\]
    Hence, the Fefferman surface area measure is
    \[d\sigma_{\operatorname{F}}=m^{\frac{2}{3}}\frac{s^{\frac{2(m-1)}{3}}}{\sqrt{1-s^{2m}+m^2s^{4m-2}}}d\sigma_{\operatorname{E}}=m^{\frac{2}{3}}s^{\frac{2m+1}{3}}ds\;d\theta\; d\varphi.\] 
    Note that the Fefferman-surface area measure is rotationally invariant. Consequently, the normalised monomials $z^j w^k/\|z^j w^k\|_{L^2_{\operatorname{F}}(\pa \mathcal{D}_{2m})}$ form a complete orthonormal basis since $\mathcal{D}_{2m}$ is a complete Reinhardt domain. So, consider
    \begin{multline*}
        \int_{\pa \mathcal{D}_{2m}}|z^j w^k|^2 d\sigma_{\operatorname{F}}=m^{\frac{2}{3}}\int_{-\pi}^{\pi}\int_{-\pi}^{\pi}\int_{0}^1 (1-s^{2m})^j s^{2k}s^{\frac{2m+1}{3}}\;ds\;d\theta \;d\phi\\
        =(2\pi)^2 m^{\frac{2}{3}}\int_{0}^1 (1-s^{2m})^j s^{2k+\frac{2m+1}{3}}ds.
    \end{multline*}
    Set $t=s^{2m}$ so $ds=\frac{1}{2m}t^{\frac{1}{2m}-1}dt$.
    Hence,
    \[\|z^j w^k\|_{L^2_F(\pa \mathcal{D}_{2m})}^2=\frac{(2\pi)^2 m^{\frac{2}{3}}}{2m}\int_{0}^1 (1-t)^j t^{\frac{3k-2m+2}{3m}}dt=\frac{2\pi^2}{m^{\frac{1}{3}}}B\Big(j+1,\frac{3k-2m+2}{3m}+1\Big),\] 
    The Beta function $B(\cdot,\cdot)$ is well-defined since $3k-2m+2>-3m$ for all nonnegative integers $k$ and positive integers $m$. 
    
   For every $\zeta=(z_1,w_1),~\eta=(z_2,w_2)\in \mathcal{D}_{2m}$, it follows from \eqref{series_exp} and the above calculations that
 \begin{multline*}
    S_{\mathcal{D}_{2m}}(\zeta,\eta)=\frac{m^{\frac{1}{3}}}{2\pi^2}\sum_{j,k\geq 0}\frac{(z_1 \ov{z}_2)^j (w_1 \ov{w}_2)^k}{B(j+1,\frac{3k+m+2}{3m})}\\
    =\frac{m^{\frac{1}{3}}}{2\pi^2}\sum_{k\geq 0}\frac{(w_1 \ov{w}_2)^k}{\Ga(\be_k)}\sum_{j\geq 0}\frac{\Ga(j+\be_k+1)}{j!}(z_1 \ov{z}_2)^j
\end{multline*}
where $\beta_k=\frac{3k+m+2}{3m}$. Note that 
\[\Ga(j+\be_k+1)=(\be_k+1)_j\Ga(\be_k+1),\] 
where $(x)_j$ is Pochhanmer. So, 
\[\sum_{j=0}^\infty \frac{\Ga(j+\be_k+1)}{j!}(z_1\ov z_2)^j=\Ga(\be_k+1)(1-z_1\ov z_2)^{-\be_k-1}.\] 
This implies that for every $\zeta=(z_1,w_1),~\eta=(z_2,w_2)\in \mathcal{D}_{2m}$, we have 
\begin{multline*}
    S_{\mathcal{D}_{2m}}(\zeta,\eta)=\frac{m^{\frac{1}{3}}}{2\pi^2}\sum_{k\geq 0}\frac{(w_1 \ov{w}_2)^k}{\Ga(\be_k)}\Ga(\be_k+1)(1-z_1\ov z_2)^{-\be_k-1}\\=\frac{m^{\frac{1}{3}}}{2\pi^2}(1-z_1\ov z_2)^{-\frac{4m+2}{3m}}\sum_{k\geq 0}\frac{3k+m+2}{3m}X^k
\end{multline*}
where $X=w_1\ov w_2/(1-z_1\ov z_2)^{\frac{1}{m}}$ with $|X|<1$ from (\ref{equat1}). 

Since 
\[\sum_{k\geq 0}X^k=\frac{1}{1-X}\quad \text{and}\quad X\sum_{k\geq 1}kX^{k-1}= \frac{X}{(1-X)^2},\] 
it follows that
\begin{equation*}
    \sum_{k\geq 0}\frac{3k+m+2}{3m}X^k=\frac{1}{3m}\frac{X(1-m)+(m+2)}{(1-X)^2}.
\end{equation*}
Therefore,
\begin{multline*}
   S_{\mathcal{D}_{2m}}(\zeta,\eta)=\frac{m^{\frac{1}{3}}}{6m\pi^2}(1-z_1\ov z_2)^{-\frac{4m+2}{3m}}\frac{m+2+(1-m)\frac{w_1\ov w_2}{(1-z_1\ov z_2)^{\frac{1}{m}}}}{\Big(1-\frac{w_1\ov w_2}{(1-z_1\ov z_2)^{1/m}}\Big)^2}\\
   =\frac{m^{\frac{1}{3}}}{2\pi^2}\frac{(m+2)(1-z_1 \ov z_2)^{\frac{1}{m}}+(1-m)w_1\ov w_2}{3m(1-z_1\ov z_2)^{\frac{4m-1}{3m}}((1-z_1 \ov z_2)^{\frac{1}{m}}-w_1\ov w_2)^2}.
\end{multline*}
Thus, the proof of Theorem \ref{feff-sgo-egg-doms} is completed.
\end{proof}
As a consequence, we have the following results.
\begin{proposition}\label{deno-feff-sgo-ker}
    For $m\in\mathbb{Z}^+$, let $\mathcal{D}_{2m}=\{(z,w)\in\mathbb{C}^2:|z|^2+|w|^{2m}<1\}$. The Fefferman--Szeg\H{o} kernel $S_{\mathcal{D}_{2m}} (\zeta,\eta)$ blows up exactly on the boundary diagonal $E=\{(\zeta,\eta)\in\partial\mathcal{D}_{2m}\times\partial\mathcal{D}_{2m}:\zeta=\eta\}$. 
\end{proposition}
\begin{proof}
It suffices to show that for $\zeta,\eta\in\overline{\mathcal{D}}_{2m}$, \[(1-z_1\overline{z}_2)^{\frac{4m-1}{3m}}((1-z_1\overline{z}_2)^{\frac{1}{m}}-w_1\overline{w}_2)^2\] vanishes if and only if $\zeta=(z_1,w_1),\eta=(z_2,w_2)\in E$. 

\medskip

Thus, we consider the two cases.

\textbf{Case I:} If $1-z_1\bar{z}_2=0$, then $|z_1||z_2|=1$. Since \(|z_i|\le1\) we have $|z_1|=|z_2|=1$. The domain condition forces $w_i=0$, thus $\zeta=(z_1,0)$, $\eta=(z_2,0)$ and $\zeta,\eta\in\partial \mathcal{D}_{2m}$.
In addition, \(z_1\bar{z}_2=1\) gives \(z_1=z_2\). Hence \(\zeta=\eta\in\partial \mathcal{D}_{2m}\).

\medskip

\textbf{Case II:} If \((1-z_1\bar{z}_2)^{\frac{1}{m}}=w_1\bar{w}_2\), then $1=z_1\bar{z}_2+w_1^m\bar{w}_2^m$. Set \(P_1=(z_1,w_1^m), P_2=(z_2,w_2^m)\); we have \(P_1\cdot P_2=1\). By the Cauchy--Schwarz inequality,
\[
1=|P_1\cdot P_2|\le|P_1||P_2|\le 1,
\]
so equality holds. Thus \(|P_1|=|P_2|=1\), hence \(\zeta,\eta\in\partial \mathcal{D}_{2m}\) and \(P_1=\lambda P_2\) with \(\lambda>0\). Hence, \(\lambda=1\), which implies \(z_1=z_2\) and \(w_1^m=w_2^m\). The boundary condition together with \(w_1\bar{w}_2=(1-|z_1|^2)^{1/m}\ge0\) forces \(w_1=w_2\). Consequently, \(\zeta=\eta\in\partial \mathcal{D}_{2m}\).

In both cases the product vanishes only when $\zeta,\eta\in E$, which completes the proof.
\end{proof}

\begin{corollary}\label{Sgoker_ext_holo}
     For $m\in\mathbb{Z}^+$, let $\mathcal{D}_{2m}=\{(z,w)\in\mathbb{C}^2:|z|^2+|w|^{2m}<1\}$. For each $\eta\in \mathcal{D}_{2m}$, the Fefferman--Szeg\H o kernel  $S_{\mathcal{D}_{2m}}(\zeta,\eta)$ extends holomorphically to a neighbourhood of $\overline{\mathcal{D}}_{2m}$.
\end{corollary}
\begin{proof}
    This follows from the explicit formula of $S_{\mathcal{D}_{2m}}(\zeta,\eta)$ derived in Theorem \ref{feff-sgo-egg-doms} and Proposition \ref{deno-feff-sgo-ker}.
\end{proof}
\begin{corollary}\label{density}
    The space of holomorphic functions on $\overline{\mathcal{D}}_{2m}$ is dense in the Fefferman-Hardy space $\operatorname{H}^2_{\operatorname{F}}(\partial \mathcal{D}_{2m})$.
\end{corollary}
\begin{proof}
    This is the consequence of Corollary \ref{Sgoker_ext_holo}. 
\end{proof}
\begin{corollary}\label{m=1-Feff-sgoker-ball}
    For $m=1$, we have
    \[S_{\mathbb B^2}(\zeta,\eta)=\frac{1}{2\pi^2}\frac{1}{(1-z_1\ov z_2-w_1\ov w_2)^2.}\]
\end{corollary}
\begin{proposition}\label{bdry-asym-sgo}
For $m\in\mathbb Z^+$, let $\mathcal{D}_{2m}=\{(z,w)\in\mathbb{C}^2~: ~\rho(z,w)=|z|^2+|w|^{2m}-1<0\}$. Then for every $\xi\in\partial \mathcal{D}_{2m}$, we have
\begin{equation}\label{lim}
   \lim_{\substack{(z,w)\to \xi \\ w\neq 0}}\frac{\rho(z,w)^{2}\,S_{\mathcal{D}_{2m}}(z,w)}{\Big(\det(\rho_{i\bar j}(z,w))\Big)^{2/3}} = \frac{1}{2\pi^{2}}.
\end{equation}
\end{proposition}
\begin{proof}
Set $u=1-|z|^2$, $v=|w|^2$ so that \(-\rho=u-v^m\). Assume $v>0$ (i.e. $w\neq0)$, and expand the kernel for sufficiently small $-\rho$. As $-\rho \to 0^+$, we have the following asymptotic expansions:
\begin{align*}
    u^{\frac{1}{m}}-v &= \frac{1}{m}v^{1-m}(-\rho)+O(\rho^2),\\
    u^{\frac{4m-1}{3m}}&=v^{\frac{4m-1}{3}}+O(-\rho),\\
    (m+2)u^{\frac{1}{m}}+(1-m)v&=3v+\frac{m+2}{m}v^{1-m}(-\rho)+O(\rho^2).
\end{align*}
Substituting these asymptotic expansions into the kernel yields:
\[
S_{\mathcal{D}_{2m}}(z,w)=\frac{m^{\frac{1}{3}}}{2\pi^2}\,
\frac{3v+\frac{m+2}{m}v^{1-m}(-\rho)+O(\rho^2)}
{3m\bigl(v^{\frac{4m-1}{3}}+O(-\rho)\bigr)\bigl(\frac{1}{m}v^{1-m}(-\rho)+O(\rho^2)\bigr)^2}.
\]
The leading term in the denominator is
\[
 \frac{3}{m}v^{\frac{4m-1}{3}+2-2m}\rho^2.
\]
Therefore
\[
\rho^2S_{\mathcal{D}_{2m}} \sim \frac{m^{\frac{4}{3}}}{2\pi^2}\,|w|^{\frac{4(m-1)}{3}},
\]
where the symbol $\sim$ means the ratio tends to $1$ as $\rho\to0^+$.

Next, note that
\[
\bigl(\det(\rho_{i\bar j})\bigr)^{\frac{2}{3}}= m^{\frac{4}{3}}|w|^{\frac{4(m-1)}{3}}.
\]
This yields, for any boundary point $\xi\in\partial \mathcal{D}_{2m}$, the limit over points $(z,w)\to \xi$ with $w\neq 0$ exists and equals $\frac{1}{2\pi^2}$.
\end{proof}
\begin{remark}
    The point of this result is that this well known formula in the strongly
pseudoconvex case extends to the points where $\det (\rho_{i \overline j})$ vanishes. We refer to \cite[P.~261]{a78} for the analogous Bergman kernel result.
\end{remark}
\subsection{A natural biholomorphic invariant}
We now recall the diagonal values of the Bergman kernel of $\mathcal{D}_{2m}$ as follows
\begin{equation}\label{diag_bergker}
    K_{\mathcal{D}_{2m}}(z,w)=
\frac{(m+1)(1-|z|^2)^{\frac{1}{m}}-(m-1)|w|^2}
{m\pi^2\Big((1-|z|^2)^{\frac{1}{m}}-|w|^2\Big)^3(1-|z|^2)^{\frac{2m-1}{m}}}
\end{equation}
Using $K_{\mathcal{D}_{2m}}$ and $S_{\mathcal{D}_{2m}}$, we define the biholomorphic invariant as follows
\[
SK_{\mathcal{D}_{2m}}(z,w)=\frac{S_{\mathcal D_{2m}}(z,w)^{3}}{K_{\mathcal{D}_{2m}}(z,w)^2}\qquad\text{ for all }(z,w)\in\mathcal{D}_{2m}.
\]
Indeed, if a biholomorphism $F:\mathcal{D}_{2m}\to \Omega$ satisfies Definition \ref{biholo_cond}, then
\[
SK_{\Omega}(F(z),F(w))=SK_{\mathcal{D}_{2m}}(z,w)\qquad\text{ for all }(z,w)\in\mathcal{D}_{2m}.
\]
For each $(z,w)\in\mathcal{D}_{2m}$, define $\tau(z,w) := |w|^2/(1-|z|^2)^{\frac{1}{m}}$. Upon substituting  \eqref{diag_bergker} and \eqref{diag_sgoker} into $SK_{\mathcal{D}_{2m}}$, a straightforward simplification yields
\begin{equation}\label{SK:cptform}
   SK_{\mathcal{D}_{2m}}(z,w)=\frac{1}{6^3\pi^2}\cdot\frac{\big((m+2)-(m-1)\tau(z,w)\big)^3}{ \big((m+1)-(m-1)\tau(z,w)\big)^2}.
\end{equation}

\begin{proposition}\label{SK-lim}
For $m\in\mathbb{Z}^+$, let $\mathcal{D}_{2m}=\{(z,w)\in\mathbb C^2: |z|^2+|w|^{2m}<1\}$. We have
\begin{itemize}
    \item [(i)] If $\xi=(0,e^{i\theta})\in\partial \mathcal{D}_{2m}$ with $\theta\in(-\pi,\pi]$, then
\[
\lim_{(z,w)\to \xi} SK_{\mathcal{D}_{2m}}(z,w)=\frac{1}{32\pi^2}.
\]
\item [(ii)] If $\xi=(e^{i\theta},0)\in\partial \mathcal{D}_{2m}$ with $\theta\in(-\pi,\pi]$, then 
    \[
    \lim_{\substack{(z,w)\to \xi \\ |w|^{2m}=o(1-|z|^2-|w|^{2m})}}SK_{\mathcal{D}_{2m}}(z,w)=\frac{1}{6^3\pi^2}\frac{(m+2)^3}{(m+1)^2}.
    \]
\end{itemize}
\end{proposition}
\begin{proof}
In view of \eqref{SK:cptform}, it suffices to determine the limiting behavior of $\tau(z,w)=|w|^2/(1-|z|^2)^{\frac{1}{m}}$
in each case.

For (i), as $(z,w)\to(0,e^{i\theta})$, we have $|z|\to 0$ and $|w|\to 1$. Hence $\tau(z,w)\to 1$. Substituting $\tau=1$ into \eqref{SK:cptform} gives
\[
\frac{1}{6^3\pi^2}\cdot \frac{3^3}{2^2}=\frac{1}{32\pi^2}.
\]

For (ii), recall that $|\rho(z,w)|=1-|z|^2-|w|^{2m}>0$. The condition $|w|^{2m}=o\big(|\rho|\big)$ implies that $1-|z|^2=|\rho|+|w|^{2m}\sim |\rho|$ as $\rho\to 0$. Therefore,
\[
\tau(z,w)=\frac{|w|^2}{(1-|z|^2)^{\frac{1}{m}}}
\sim\frac{o(|\rho|^{\frac{1}{m}})}{|\rho|^{\frac{1}{m}}}\to 0.
\]
Substituting $\tau=0$ into \eqref{SK:cptform} yields the desired constant.

\end{proof}
\begin{remark}\label{SK-rem}
The two limits appearing in the preceding corollary coincide precisely when \(m=1\). In this exceptional case the domain is the unit ball \(\mathbb{B}^{2}\), and the invariant \(SK_{\mathbb{B}^{2}}\) is identically constant; see \cite{bl14}.
\end{remark}
\begin{proposition}\label{SK:bdd}
For $m\in\mathbb Z^+$, let $\mathcal{D}_{2m}=\{(z,w)\in\mathbb C^2: |z|^2+|w|^{2m}<1\}$. The function $SK_{\mathcal{D}_{2m}}$ is uniformly bounded above and below on $\mathcal{D}_{2m}$.
\end{proposition}
\begin{proof}
We consider a continuous function
\[f(x)=\frac{1}{6^3\pi^2}\,\frac{\bigl((m+2)+(1-m)x\bigr)^3}{\bigl((m+1)-(m-1)x\bigr)^2},\qquad\text{ for all }x\in [0,1].\]
 For $m=1$, it is clear that $f\equiv 1/(32\pi^2)$. So we only need to focus on $m> 1$.

For $m>1$, note that $f$ is a strictly decreasing function.   Therefore, $f$ cannot be identically constant. Also,
$$
\sup_{x\in [0,1]}f(x)=f(0)=\frac{1}{6^3\pi^2}\frac{(m+2)^3}{(m+1)^2},
$$
and
$$
\inf_{x\in [0,1]} f(x)=f(1)=\frac{1}{32\pi^2}.
$$
Consequently,
\[\frac{1}{32\pi^2}<SK_{\mathcal{D}_{2m}}(z,w)\le\frac{1}{6^3\pi^2}\frac{(m+2)^3}{(m+1)^2},\]
$\text{ for all }(z,w)\in\mathcal{D}_{2m}$. Thus, the upper bound of $SK_{\mathcal{D}_{2m}}$ is attained on the set $\{w=0\}$ (where $\tau=0$) and the lower bound is strict for $m>1$ (approached as $(z,w)\to(0,e^{i\theta})$). hence, the proof is completed.
\end{proof}
\section{Rigidity of the Fefferman--Szeg\H o metric}
 In this section, we compare three such metrics on the egg domain
\[
\mathcal D_{2m}:=\{(z,w)\in\mathbb C^2: |z|^2+|w|^{2m}<1\},\qquad m\in\mathbb Z^+.
\]
These are: the Fefferman--Szeg\H o metric \(g_{\operatorname{FS}}^{\mathcal D_{2m}}\) associated with the Fefferman surface area measure, the classical Bergman metric \(g_{\operatorname{B}}^{\mathcal D_{2m}}\), and the complete K\"ahler metric \(g_m^{\mathcal D_{2m}}\) constructed from the alternative defining function
\[
\rho_m(z,w):=(1-|z|^2)^{\frac{1}{m}}-|w|^2
\]
by setting \(g_m^{\mathcal D_{2m}}:=-\partial_i\partial_{\bar j}\log \rho_m\). The metric \(g_m^{\mathcal D_{2m}}\) was introduced by Seo \cite{Seo12} and later studied by Sha \cite{sha26}; it is natural to compare it with the Fefferman--Szeg\H o metric.

The main result of this section shows that all K\"ahler--Einstein rigidity phenomena and proportionality relations among these metrics occur exactly in the case \(m=1\), i.e., when \(\mathcal D_{2m}\) is the unit ball $\mathbb B^2$.

\begin{proof}[Proof of Theorem \ref{main:egg}]
    We prove that each of the four statements is equivalent to \(m=1\).

    \medskip
    \noindent\textbf{(i) \(\Longleftrightarrow\) (iv).}
    By Proposition \ref{FS-KE}, the Fefferman--Szeg\H o metric is K\"ahler--Einstein if and only if \(m=1\).

    \medskip
    \noindent\textbf{(ii) \(\Longleftrightarrow\) (iv).}
    If \(m=1\), then \(\mathcal D_{2m}\) is the unit ball, and the Fefferman--Szeg\H o and Bergman metrics coincide up to a constant factor (see Remark \ref{SK-rem}). Conversely, if \(g_{\operatorname{FS}}^{\mathcal D_{2m}}=\lambda g_{\operatorname{B}}^{\mathcal D_{2m}}\) for some \(\lambda>0\), then Proposition \ref{FS-gB} forces \(m=1\).

    \medskip
    \noindent\textbf{(iii) \(\Longleftrightarrow\) (iv).}
    If \(m=1\), then \(\mathcal D_{2m}\) is the unit ball, and the Fefferman--Szeg\H o and \(g_m\) metrics coincide up to a constant factor. Conversely, if \(g_{\operatorname{FS}}^{\mathcal D_{2m}}=\lambda g_m^{\mathcal D_{2m}}\) for some \(\lambda>0\), then Proposition \ref{FS-gm} forces \(m=1\).

    \medskip
    Since all three statements are equivalent to (iv), they are equivalent to each other. This completes the proof. 
\end{proof}
It remains to establish the propositions used in the proof above. We first record the explicit expressions for the Fefferman--Szeg\H o metric and its Ricci curvature on the distinguished slice \(\{z=0\}\). This is sufficient because every point of \(\mathcal D_{2m}\) can be moved to this slice by an automorphism. 

Indeed, the group of holomorphic automorphisms \(\operatorname{Aut}(\mathcal D_{2m})\) contains the maps
\begin{equation}\label{auto_eggdoms}
    F(z,w)=\left(\frac{z-z_0}{1-\bar z_0 z},\;
\mu\,\frac{(1-|z_0|^2)^{\frac{1}{2m}}}{(1-\bar z_0 z)^{\frac{1}{m}}}\,w\right),\qquad\qquad
\mu=\begin{cases}
\dfrac{|w_0|}{w_0},& w_0\neq0,\\
1,& w_0=0,
\end{cases}
\end{equation}
which send any point \((z_0,w_0)\) to \((0, |w_0|(1-|z_0|^2)^{-\frac{1}{2m}})\). Since the Fefferman--Szeg\H o metric is invariant under such biholomorphisms (see Definition \ref{biholo_cond}), it suffices to evaluate at points of the form \((0,w)\).
\subsection{The Fefferman--Szeg\H o metric and its Ricci curvature on $\mathcal{D}_{2m}$}

\begin{proof}[Proof of Theorem \ref{FSmetric-Ricci}]
   Throughout the proof, $r$ denotes the constant $r = (m-1)/(m+2)$ defined in the theorem statement. Set
\begin{equation*}
   a = |z|^2,\quad v = |w|^2,\quad I = (1-a)^{\frac{1}{m}},\quad M= I- v,\quad N = (m+2)I- (m-1)v.
\end{equation*} 
The diagonal values of $S_{\mathcal{D}_{2m}}$ is given by
    \[
  S_{\mathcal D_{2m}}(z,w) = c\,\frac{N}{I^{\frac{4m-1}{3}} M^2},
    \]
    where \(c>0\) is a constant, so we define
    \begin{equation}\label{Psi-av}
         \Psi(a,v):=\log  S_{\mathcal D_{2m}}(z,w)
    = \log c + \log N - \frac{4m-1}{3}\log I - 2\log M.
    \end{equation}

    \medskip
    
    \noindent\textbf{The \((1,\bar 1)\)-component of $g_{\operatorname{FS}}^{\mathcal{D}_{2m}}$.}
 For a function of $a,v,~\partial_{z\overline{z}}|_{a=0}=\partial_{a}|_{a=0}$. Differentiating \eqref{Psi-av} with respect to $a$ and evaluating at $a=0$ yields
    
    \begin{align}\label{a-derPsi-0}
    \left.\partial_a \Psi\right|_{a=0}
    &= \frac{1-r}{1+2r}\left( -\frac{1}{1-rv} + \frac{1+3r}{1-r} + \frac{2}{1-v} \right).
    \end{align}
    To simplify \eqref{a-derPsi-0}, define
    \[
    t := \frac{1-v}{1-rv}.
    \]
 Using \(\frac{1}{1-rv} = \frac{t}{1-v}\) and \(1-v = \frac{t(1-r)}{1-rt}\), \eqref{a-derPsi-0} simplifies to
    
    \[
g_{1\bar 1}(0,w)
    = \left.\partial_a \Psi\right|_{a=0}
    = \frac{2 + rt + rt^2}{t(1+2r)}
    = \frac{\alpha(t)}{t(1+2r)}
    \]
    with \(\alpha(t):=2+rt+rt^2\).

    \medskip
    
    \noindent\textbf{The \((2,\bar 2)\)-component of $g_{\operatorname{FS}}^{\mathcal{D}_{2m}}$.} At \(a=0\), we have
    \[
    g_{2\bar 2}(0,w) = \left.\partial_v \Psi\right|_{a=0} + v\left.\partial_{vv} \Psi\right|_{a=0}.
    \]
    From \eqref{Psi-av}, 
    \begin{equation*}
        \left.\partial_v \Psi\right|_{a=0}
    =  -\frac{r}{1-rv} + \frac{2}{1-v}\qquad \left.\partial_{vv} \Psi\right|_{a=0}
    =-\frac{r^2}{(1-rv)^2} + \frac{2}{(1-v)^2}. 
    \end{equation*}
   Thus, using \(\frac{1}{1-rv}=\frac{t}{1-v}\), gives
    \begin{align}\label{wbarw-feff-sgo}
    g_{2\bar 2}(0,w)
    &= \frac{2-rt}{1-v} + \frac{v(2-r^2t^2)}{(1-v)^2}. 
    \end{align}
  Using $ 1-v = \frac{t(1-r)}{1-rt}\text{ and } v = \frac{1-t}{1-rt}$,
     \eqref{wbarw-feff-sgo} becomes
    \begin{align}\label{final-wbarw-feff-sgo}
    g_{2\bar 2}(0,w)
    &= \frac{(2-rt)(1-rt)}{t(1-r)}
       + \frac{(1-t)(1-rt)(2-r^2t^2)}{t^2(1-r)^2}.
    \end{align}
    Factoring \(\frac{1-rt}{t^2(1-r)^2}\) and using the identity 
    \[t(2-rt)(1-r)+(1-t)(2-r^2t^2)=(1-rt)(2-rt^2),\]
    we obtain
\[g_{2\bar 2}(0,w)
    = \frac{(1-rt)^2(2-rt^2)}{t^2(1-r)^2}
    = \beta(t)\frac{(1-rt)^2}{t^2(1-r)^2}
    \]
    with \(\beta(t):=2-rt^2\).

    \medskip
    
    \noindent\textbf{The mixed components of $g_{\operatorname{FS}}^{\mathcal{D}_{2m}}$.}
    Finally, since $ \left.\partial_{z\bar w}\right|_{z=0} = \bar z w\,\partial_{av}^2|_{z=0} =0,$ we have
    \[
    g_{1\bar 2}(0,w) = \left.\partial_{z\bar w}\log S_{\mathcal D_{2m}}\right|_{z=0} = 0.
    \] This implies $ g_{2\bar 1}(0,w)=0$.
    Thus, the proof of the metric formula is completed.

    \medskip

    We finally compute the Ricci curvature of $g_{\operatorname{FS}}^{\mathcal{D}_{2m}}$. 

    \noindent\textbf{Derivatives of $\Psi$ at $a=0$.}
 For the Ricci tensor computation, we require the following derivatives of $\Psi$ evaluated at $a=0$:
    \begin{equation}\label{Der-Psi}
    \begin{array}{lll}
    \Psi_a=\dfrac{\alpha(t)}{t(1+2r)}, &
    \Psi_v=\dfrac{(2-rt)(1-rt)}{t(1-r)}, &
    \Psi_v+v\Psi_{vv}=\dfrac{\beta(t)(1-rt)^2}{t^2(1-r)^2},\\[8pt]
    \Psi_{av}=\dfrac{(1-rt)^2(2-rt^2)}{(2r+1)(1-r)t^2}, &
    \Psi_{avv}=\dfrac{2(1-rt)^3(2-r^2t^3)}{(2r+1)(1-r)^2t^3}, &
    \Psi_{aa}=\dfrac{A(t)+2\alpha(t)}{(2r+1)^2t^2},\\[8pt]
    \Psi_{vvv}=\dfrac{2(1-rt)^3(2-r^3t^3)}{t^3(1-r)^3}, &
    \Psi_{vvvv}=\dfrac{6(1-rt)^4(2-r^4t^4)}{t^4(1-r)^4}.
    \end{array}  \end{equation}

    Here and below, $\Psi_a:=\partial_a\Psi|_{a=0}$, $\Psi_{aa}:=\partial_{a}^2\Psi|_{a=0}$ and so on.

    \medskip
    \noindent\textbf{The $(1,\bar 1)$-component of the Ricci curvature.}
   From \eqref{Ric-ibarj}, since the metric is diagonal,
   \[\operatorname{Ric}_{1\bar 1}(0,w)=\sum_{k=1}^2g^{k\bar k}(0,w)R_{1\bar 1 k\bar k}(0,w).\]  Using \eqref{Der-Psi} into \eqref{Ric-ibarj}, we find
    \[
    R_{1\bar 11\bar 1}(0,w)=-2\Psi_{aa},
    \]
    and
    \begin{equation}\label{R_1bar12bar2}
        R_{1\bar 12\bar 2}(0,w)
    =-\Psi_{av}-v\Psi_{avv}+v(\Psi_{av})^2\frac{1}{\Psi_a}.
    \end{equation}
    Hence
    \begin{equation}\label{zbarz-Ric}
        \operatorname{Ric}_{1\bar1}(0,w)
    =-2\frac{\Psi_{aa}}{\Psi_a}
    +\frac{1}{\Psi_v+v\Psi_{vv}}
    \left(-\Psi_{av}-v\Psi_{avv}+v(\Psi_{av})^2\frac{1}{\Psi_a}\right).
    \end{equation}
    
    Substituting \eqref{Der-Psi} into \eqref{zbarz-Ric}, the first term gives
    \[
    -2\frac{\Psi_{aa}}{\Psi_a}
    =
    -\frac{2A(t)+4\alpha(t)}{t\alpha(t)(1+2r)},
    \]
    where $A(t)=-r^{2}t^{4}+(3r^{2}+2r)t^{3}+2r^{2}t^{2}-2$. 
    
    For the second term, we substitute explicitly:
    \[
    \begin{aligned}
    &\frac{1}{\Psi_v+v\Psi_{vv}}
    \left(-\Psi_{av}-v\Psi_{avv}+v(\Psi_{av})^2\frac{1}{\Psi_a}\right) \\
    &= \frac{t^2(1-r)^2}{(1-rt)^2(2-rt^2)}
    \left[
    -\frac{(1-rt)^2(2-rt^2)}{(2r+1)(1-r)t^2}
    -\frac{2(1-t)(1-rt)^2(2-r^2t^3)}{(2r+1)(1-r)^2t^3}
    \right. \\
    &\qquad\qquad\qquad\qquad\qquad
    \left.
    +\frac{(1-t)(1-rt)^3(2-rt^2)^2}{(1-r)^2 t^3 \alpha(t) (1+2r)}
    \right].
    \end{aligned}
    \]
    Putting the bracket over the common denominator
    \[
    (2r+1)(1-r)^2 t^3 \alpha(t) (1+2r),
    \]
    and simplifying the numerator, we obtain
    \[
    \frac{1}{\Psi_v+v\Psi_{vv}}
    \left(-\Psi_{av}-v\Psi_{avv}+v(\Psi_{av})^2\frac{1}{\Psi_a}\right)
    =
    \frac{B(t)}{(2r+1)t(2-rt^2)}
    +\frac{C(t)}{t\alpha(t)(1+2r)},
    \]
    where $ B(t)=-2r^{2}t^{4}+(r+r^{2})t^{3}+2(1+r)t-4\text{ and }
    C(t)=(1-t)(1-rt)(2-rt^2)$. Therefore,
    \begin{multline}\label{1bar1-Ricci}
         \operatorname{Ric}_{1\bar1}(0,w)
    =
    -\frac{2A(t)+4\alpha(t)}{t\alpha(t)(1+2r)}
    +\frac{B(t)}{(2r+1)t(2-rt^2)}
    \\+\frac{C(t)}{t\alpha(t)(1+2r)}
=:I+II+III. 
    \end{multline}

    \medskip
    
  \noindent\textbf{The $(2,\bar 2)$-component of the Ricci curvature.}  
From \eqref{Ric-ibarj}, since the metric is diagonal,
\[
\operatorname{Ric}_{2\bar 2}(0,w) = \sum_{k=1}^2 g^{k\bar k}(0,w) R_{2\bar 2 k\bar k}(0,w).
\]

For this component, we note that \(R_{2\bar 2 1\bar 1}(0,w)=R_{1\bar 1 2\bar 2}(0,w)\) and, 
\[
R_{2\bar 22\bar 2}(0,w)
=-g_{2\bar 22\bar 2}(0,w)+g_{22\bar2}(0,w)\,g^{2\bar2}(0,w)\,g_{2\bar2\bar2}(0,w).
\]

A direct computation using \eqref{Psi-av} gives
\[
g_{2\bar 22\bar 2}(0,w)=2\Psi_{vv}+4v\Psi_{vvv}+v^2\Psi_{vvvv},
\]
and
\[
g_{22\bar2}(0,w)\,g^{2\bar2}(0,w)\,g_{2\bar2\bar2}(0,w)
=\frac{1}{\Psi_v+v\Psi_{vv}}
\Bigl(4v(\Psi_{vv})^2+4v^2\Psi_{vv}\Psi_{vvv}+v^3(\Psi_{vvv})^2\Bigr).
\]

Substituting \eqref{Der-Psi} into \(\operatorname{Ric}_{2\bar2}(0,w)\) and simplifying yields
\begin{equation}\label{Ric-2bar2}
    \operatorname{Ric}_{2\bar2}(0,w)=D(t)+E(t)+F(t),
\end{equation}
where
\[
D(t)=
\frac{-2t^2(1-r)^2(2-r^2t^2)-8t(1-r)(1-t)(2-r^3t^3)-6(1-t)^2(2-r^4t^4)}
{t^2(1-r)^2(2-rt^2)}
\]
is the contribution from \(-g^{2\bar2}(0,w)g_{2\bar 22\bar2}(0,w)\),
\[
E(t)=
\frac{4(1-t)(1-rt)(2-r^2t^3)^2}
{t^2(1-r)^2(2-rt^2)^2}
\]
is the contribution from \(g^{2\bar2}(0,w)g_{22\bar2}(0,w)g^{2\bar2}(0,w)g_{2\bar2\bar2}(0,w)\), and
\[
F(t)=
-\frac{(1-rt)^2(2-rt^2)}{(1-r)t\alpha(t)}
-\frac{2(1-t)(1-rt)^2(2-r^2t^3)}
{(1-r)^2t^2\alpha(t)}
+\frac{(1-t)(1-rt)^3(2-rt^2)^2}
{(1-r)^2t^2\alpha(t)^2}
\]
is the contribution from \(g^{1\bar1}(0,w)R_{2\bar21\bar1}(0,w)\) using \eqref{R_1bar12bar2}.

\medskip
\noindent\textbf{The mixed components of the Ricci curvature.} 
It remains to verify that the off-diagonal Ricci components vanish. 
From \eqref{Ric-ibarj}, since the metric is diagonal,
\[
\operatorname{Ric}_{1\bar 2}(0,w) = g^{1\bar1}(0,w)R_{1\bar 2 1\bar1}(0,w) + g^{2\bar2}(0,w)R_{1\bar 2 2\bar2}(0,w).
\]
Using \eqref{Ric-ibarj}, the required curvature components are
\[
R_{1\bar 2 1\bar 1}
= -g_{1\bar 2 1\bar 1} + g_{1\bar 1 1}g^{1\bar1}g_{1\bar 2\bar1}
+ g_{1\bar 2 1}g^{2\bar2}g_{2\bar 2\bar1},
\]
and
\[
R_{1\bar 2 2\bar 2}
= -g_{1\bar 2 2\bar 2} + g_{1\bar 1 2}g^{1\bar1}g_{1\bar 2\bar2}
+ g_{1\bar 2 2}g^{2\bar2}g_{2\bar 2\bar2}.
\]
At \(a=0\), direct differentiation gives the following vanishing derivatives:
\[
g_{1\bar1 1}=0,\quad
g_{1\bar2 1}=0,\quad g_{1\bar2 2}=0,\quad
g_{1\bar2\bar2}=0,\quad g_{1\bar2 1\bar1}=0,\quad g_{1\bar2 2\bar2}=0.
\]
For instance, \(g_{1\bar1 1}=\partial_z g_{1\bar1} = 2\bar z \Psi_{aa} + \bar z a \Psi_{aaa}=0\) at \(a=0\); the remaining identities follow similarly. 
Substituting these into the expressions for \(R_{1\bar 2 1\bar1}\) and \(R_{1\bar 2 2\bar2}\) yields
\[
R_{1\bar 2 1\bar1}=0,\qquad R_{1\bar 2 2\bar2}=0.
\]
Hence \(\operatorname{Ric}_{1\bar 2}(0,w)=0\). By symmetry of the Ricci tensor, \(\operatorname{Ric}_{2\bar 1}(0,w)=0\). Thus the Ricci curvature is fully diagonal.    
\end{proof}

Throughout the remainder of this Section, unless otherwise stated, we use the notation of Lemma \ref{FSmetric-Ricci}.
\begin{proposition}\label{FS-KE}
    The Fefferman Szeg\H o metric $g_{\operatorname{FS}}^{\mathcal{D}_{2m}}$ is K\"ahler-Einstein if and only if $m=1$.
\end{proposition}
\begin{proof}
  If \(m=1\), then \(\mathcal D_2=\mathbb B^2\), and the Fefferman--Szeg\H o metric coincides (up to a constant factor) with the Bergman metric, which is well known to be K\"ahler--Einstein on unit ball. 
      Conversely, suppose \(g_{\operatorname{FS}}^{\mathcal D_{2m}}\) is K\"ahler--Einstein. Then there exists a constant \(\lambda\in\mathbb R\) such that
    \[
    \operatorname{Ric}_{2\bar2}(0,w)=\lambda\,g_{2\bar2}(0,w) \tag{KE}
    \]
    for all \(|w|<1\). We will evaluate this identity in two limiting regimes.

\medskip

\noindent\textbf{Boundary asymptotics as $t\to0$ (i.e, $|w|\to 1$):} From Lemma \ref{FSmetric-Ricci}, the metric $(2,\bar 2)$-component has the expansion
\begin{equation}\label{g22}
    g_{2\bar2}(0,w)=\frac{(1-tr)^2(2-rt^2)}{t^2(1-r)^2}=\frac{2}{(1-r)^2}\frac{1}{t^2}+O(t^{-1})\quad\text{as}\quad t\to 0.
\end{equation}
This implies that we can only focus on the term $t^{-2}$. A routine calculation yields the following asymptotic expansions
\begin{align*}
        D(t)&=-\frac{6}{(1-r)^2}\frac{1}{t^2}+O(t^{-1}),\\
        E(t)&=\frac{4}{(1-r)^2}\frac{1}{t^2}+O(t^{-1}),\\
        F(t)&=-\frac{1}{(1-r)^2}\frac{1}{t^2}+O(t^{-1}).
\end{align*}
From \eqref{Ric-2bar2}, we have 
\begin{align}\label{add}
  \operatorname{Ric}_{2\bar2}(0,w)=-\frac{3}{(1-r)^2}\frac{1}{t^2}+O(t^{-1}).
\end{align}
Hence, combining  (\ref{g22}) and (\ref{add}):
\begin{align}\label{ratio-Ricci-feffsgo-metric}
    \lim_{\substack{t\to 0 \\ {\text{i.e., }|w|\to 1}}}\frac{\operatorname{Ric}_{2\bar2}(0,w)}{g_{2\bar2}(0,w)}=\lim_{t\to 0}\frac{-\frac{3}{(1-r)^2}\frac{1}{t^2}+O(t^{-1})}{\frac{2}{(1-r)^2}\frac{1}{t^2}+O(t^{-1})}=-\frac{3}{2}.
\end{align}
This implies that $\lambda=-3/2.$

\medskip

\noindent\textbf{Evaluation at $t=1$ (i.e., $w=0$):} Substituting $t=1$ into (\ref{Ric-2bar2}) and (\ref{g22}), we have
\begin{align*}
    \operatorname{Ric}_{2\bar2}(0,0)=D(1)+E(1)+F(1)
    &=-\frac{2(2-r^2)}{2-r}-\frac{(1-r)(2-r)}{2(1+r)}
\end{align*}
and 
$$g_{2\bar2}(0,0)=\frac{(2-r)(1-r)^2}{(1-r)^2}=2-r.$$
Combining the K\"ahler--Einstein condition (KE) with \(\lambda=-3/2\) yields
$$\operatorname{Ric}_{2\bar 2}(0,0)=-\frac{3}{2}g_{2\bar 2}(0,0),$$
that is
$$r^2(5-4r)=0.$$
Since \(0\le r<1\), this forces \(r=0\), which is equivalent to \(m=1\). This completes the proof.
\end{proof}

\begin{proposition}\label{coro1}
    The Fefferman-Szeg\H{o} metric $g_{\operatorname{FS}}^{\mathcal{D}_{2m}}$ is negatively pinched near every strongly pseudoconvex boundary point, i.e.,
    \begin{equation}
        \lim_{\substack{t\to 0 \\ {\text{i.e., }|w|\to 1}}}\operatorname{Ric}_{\operatorname{FS}}^{\mathcal{D}_{2m}}((0,w),X)=-\frac{3}{2}\quad \text{for all}\quad X\in\mathbb{C}^2\setminus \left\{0\right\},
    \end{equation}
\end{proposition}
\begin{proof}
Since the Ricci curvature is diagonal at \((0,w)\) and by \eqref{ratio-Ricci-feffsgo-metric}, it is enough to establish that 
    $$\lim_{\substack{t\to 0 \\ {\text{i.e., }|w|\to 1}}}\frac{\operatorname{Ric}_{1\bar1}(0,w)}{g_{1\bar1}(0,w)}=-\frac{3}{2}.$$
As $t\to 0$, a routine calculation yields the following asymptotic expansions:
\begin{align*}
        g_{1\bar 1}(0,w)&=\frac{2}{1+2r}\frac{1}{t}+O(1),\\
        I&=-\frac{2}{1+2r}\frac{1}{t}+O(1),\\
        II&=-\frac{2}{1+2r}\frac{1}{t}+O(1),\\
       III&=\frac{1}{1+2r}\frac{1}{t}+O(1).
\end{align*}
Here $I, II, \text{ and }III$ denote the three terms appearing in \eqref{1bar1-Ricci}. Then, the leading term in $\operatorname{Ric}_{1\bar1}(0,w)$ is
$$-\frac{2}{1+2r}\frac{1}{t}-\frac{2}{1+2r}\frac{1}{t}+\frac{1}{1+2r}\frac{1}{t}=-\frac{3}{1+2r}\frac{1}{t}.$$
Hence, the first part is completed. 

For any strongly pseudoconvex boundary point $(\xi_1,\xi_2)\in \partial\mathcal{D}_{2m}$ with $\xi_2\neq0$, when $(z,w)\to (\xi_1,\xi_2)$, we have 
\begin{equation}\label{trans}
    \abs{w}(1-\abs{z}^2)^{-\frac{1}{2m}}\to \abs{\xi_{2}}(1-\abs{\xi_{1}}^2)^{-\frac{1}{2m}}=1.
\end{equation}
Since Ricci curvature is invariant under the holomorphic automorphisms $F$ defined in \eqref{auto_eggdoms}, we have 
\begin{equation*}
    \operatorname{Ric}_{\operatorname{FS}}^{\mathcal{D}_{2m}}((z,w),X)=\operatorname{Ric}_{\operatorname{FS}}^{\mathcal{D}_{2m}}((0,\abs{w}(1-\abs{z}^2)^{-\frac{1}{2m}}),F_{*}X).
\end{equation*}
Hence, combining this with (\ref{trans}), 
$$\lim_{(z,w)\to (\xi_{1},\xi_{2})}\operatorname{Ric}_{\operatorname{FS}}^{\mathcal{D}_{2m}}((z,w),X)=-\frac{3}{2}.$$
This proof is completed.
\end{proof}

\subsection{Another K\"ahler metric on the egg domain}\label{another}

For \(m>1\), the function \(-\log(1-|z|^2-|w|^{2m})\) is not strictly plurisubharmonic, so the associated \((1,1)\)-form does not define a K\"ahler metric. To obtain a complete K\"ahler metric on \(\mathcal D_{2m}\) comparable to the Fefferman--Szeg\H o metric \(g_{\operatorname{FS}}^{\mathcal D_{2m}}\), we use the alternative defining function introduced by Seo \cite{Seo12}:
\[
\rho_m(z,w):=(1-|z|^2)^{\frac{1}{m}}-|w|^2.
\]
Set \(g_m^{\mathcal D_{2m}}:=-\partial_i\partial_{\bar j}\log \rho_m\). A direct computation gives
\[
g_m^{\mathcal D_{2m}}
=
\frac{(1-|z|^2)^{\frac{1}{m}-2}}{m\rho_m^2}
\begin{pmatrix}
\rho_m+\frac{1}{m}|z|^2|w|^2 & w\bar z(1-|z|^2)\\
\bar w z(1-|z|^2) & m(1-|z|^2)^2
\end{pmatrix},
\]
and
\[
\det g_m^{\mathcal D_{2m}}
=
\frac{\frac{1}{m}(1-|z|^2)^{\frac{2}{m}-2}}{\rho_m^3}.
\]
Consequently,
\[
\operatorname{Ric}(g_m^{\mathcal D_{2m}})
=
-\frac{3(1-|z|^2)^{\frac{1}{m}-2}}{m\rho_m^2}
\begin{pmatrix}
\frac{2}{3}(m-1)\rho_m^2(1-|z|^2)^{-\frac{1}{m}}+\rho_m+\frac{1}{m}|z|^2|w|^2 & w\bar z(1-|z|^2)\\
\bar w z(1-|z|^2) & m(1-|z|^2)^2
\end{pmatrix}.
\]
The formulas for \(g_m^{\mathcal D_{2m}}\) and its Ricci curvature (cf. \cite{sha26}) immediately yield the following rigidity result.

\begin{proposition}(\cite{sha26})\label{gm Kahler}
    The metric \(g_m^{\mathcal D_{2m}}\) is K\"ahler--Einstein if and only if \(m=1\).
\end{proposition}

We now compare the Fefferman--Szeg\H o metric with \(g_m^{\mathcal D_{2m}}\). When \(m=1\), the two metrics coincide up to a constant factor. Otherwise, they are not proportional, as shown below.

\begin{proposition}\label{FS-gm}
    If \(g_{\operatorname{FS}}^{\mathcal D_{2m}}=\lambda g_m^{\mathcal D_{2m}}\) for some \(\lambda>0\), then \(m=1\).
\end{proposition}

\begin{proof}
    Assume \(g_{\operatorname{FS}}^{\mathcal D_{2m}}=\lambda g_m^{\mathcal D_{2m}}\). 
    Hence
    \[
  \mathscr{F}:=\log S_{\mathcal D_{2m}}+\lambda\log \rho_m
    \]
    is pluriharmonic on \(\mathcal D_{2m}\). Both \(S_{\mathcal D_{2m}}\) and \(\rho_m\) depend only on \(|z|\) and \(|w|\), so \(\mathscr{F}\) is a radial pluriharmonic function. A radial pluriharmonic function on the complete Reinhardt domain is necessarily constant; hence \(\mathscr{F}\equiv C\) for some constant \(C\).

    Using \eqref{diag_sgoker} and the constancy of \(\mathscr{F}\) gives
    \begin{equation}\label{ratio-sgo-rhom}
        \frac{(m+2)(1-|z|^2)^{\frac{1}{m}}+(1-m)|w|^2}
    {(m+2)(1-|z|^2)^{\frac{4m-1}{3m}}\rho_m(z,w)^{2-\lambda}}
    =1 
    \end{equation}
    for all \((z,w)\in\mathcal D_{2m}\), after absorbing constants.

    Let \((z,w)\) approach a strongly pseudoconvex boundary point of $\partial\mathcal{D}_{2m}$ where the numerator in \eqref{ratio-sgo-rhom} does not vanish (e.g. a point with \(\rho_m\to 0\) but \(|w|>0\)). If \(0<\lambda<2\), the denominator tends to \(0\), so the left-hand side of \eqref{ratio-sgo-rhom} tends to \(+\infty\), a contradiction. If \(\lambda>2\), the denominator tends to \(+\infty\), so the left-hand side tends to \(0\), again a contradiction. Hence \(\lambda=2\).

    With \(\lambda=2\), set \(w=0\) in \eqref{ratio-sgo-rhom}, then it reduces to
    \[
    (1-|z|^2)^{\frac{1}{m}-\frac{4m-1}{3m}}=1
    \qquad\text{for all } |z|<1.
    \]
    Therefore \(\frac{1}{m}-\frac{4m-1}{3m}=0\), which forces \(m=1\). 
\end{proof}
\subsection{Comparison with the Bergman metric}

\begin{proposition}\label{FS-gB}
    If \(g_{\operatorname{B}}^{\mathcal D_{2m}}=\lambda g_{\operatorname{FS}}^{\mathcal D_{2m}}\) for some \(\lambda>0\), then \(m=1\).
\end{proposition}

\begin{proof}
    Assume \(g_{\operatorname{B}}^{\mathcal D_{2m}}=\lambda g_{\operatorname{FS}}^{\mathcal D_{2m}}\). 
    Hence
    \[
    \Upsilon:=\log K_{\mathcal D_{2m}}-\lambda\log S_{\mathcal D_{2m}}
    \]
    is pluriharmonic on \(\mathcal D_{2m}\).

    Using the diagonal values of $K_{\mathcal{D}_{2m}}$ and $S_{\mathcal{D}_{2m}}$ in \eqref{diag_bergker} and \eqref{diag_sgoker}, it follows that
    \begin{multline}\label{logSK}
        \Upsilon(z,w)=C_0
    +\log\bigl((m+1)I-(m-1)v\bigr)
    -\lambda\log\bigl((m+2)I+(1-m)v\bigr)
    \\+(-3+2\lambda)\log(I-v)
    +\left(-(2m-1)+\lambda\frac{4m-1}{3}\right)\log I
    \end{multline}
    where \(v=|w|^2\), \(I=(1-|z|^2)^{\frac{1}{m}}\).
    
    Restricting to \(\{z=0\}\) (so \(I=1\)) gives
    \begin{equation}\label{logSK-on-z=0}
        \Upsilon(0,w)=C_0
    +\log\bigl((m+1)-(m-1)v\bigr)
    -\lambda\log\bigl((m+2)+(1-m)v\bigr)
    +(-3+2\lambda)\log(1-v).
    \end{equation}

    Since \(\Upsilon\) is pluriharmonic, its restriction to \(\{z=0\}\) is harmonic in \(w\). The first two logarithmic terms in \eqref{logSK-on-z=0} are smooth on the closed unit disk, while
    \[
    \Delta_w\log(1-|w|^2)=-\frac{4}{(1-|w|^2)^2}\to -\infty \qquad (|w|\to 1).
    \]
    Hence the coefficient of \(\log(1-v)\) in \eqref{logSK-on-z=0} must vanish, so \(\lambda=\frac{3}{2}\).

    Thus \(\Upsilon:=\log(S^3/K^2)\) is pluriharmonic. Restricting to \(\{z=0\}\),
    \[
    \Upsilon(0,w)=\log f(|w|^2),
    \qquad
    f(x):=\frac{1}{6^3\pi^2}\,
    \frac{\bigl((m+2)+(1-m)x\bigr)^3}
    {\bigl((m+1)-(m-1)x\bigr)^2},\quad x\in[0,1).
    \]
    Since \(\Upsilon\) is harmonic in \(w\), \(\log f(|w|^2)\) is a radial harmonic function on the unit disk, hence constant. Therefore \(f\) is constant on \([0,1)\).

    But for \(m>1\),
    \[
    f'(x)
    =-\frac{1}{6^3\pi^2}\frac{(m-1)^2[(m+2)+(1-m)x]^2(1-x)}{[(m+1)-(m-1)x]^3}
    <0,
    \]
    so \(f\) is strictly decreasing. Thus \(f\) cannot be constant unless \(m=1\), which completes the proof.
\end{proof}

\section{\texorpdfstring{$L^2$-cohomology for $g_{\operatorname{FS}}^{D_{2m}}$}{L2-COHOMOLOGY FOR gFS{D2m}}}

 In this section, we prove Theorem \ref{van-L2-coho}, which is inspired from \cite{Don97}. Let $\Xi_2^k(\mathcal{D}_{2m})$ denote the space of square-integrable $k$-forms on $\mathcal{D}_{2m}$ with respect to $g_{\operatorname{FS}}^{\mathcal{D}_{2m}}$. We then have the differential complex
\[
\Xi_2^0(\mathcal{D}_{2m}) \xrightarrow{d_0} \Xi_2^1(\mathcal{D}_{2m}) \xrightarrow{d_1} \Xi_2^2(\mathcal{D}_{2m}) \xrightarrow{d_2}
\cdots \xrightarrow{d_{2n-1}}\Xi_{2}^{2n}(\mathcal{D}_{2m})\xrightarrow{d_{2n}}0\]
and the associated $L^2$-cohomology groups are defined by
\[
H_{2}^k(\mathcal{D}_{2m}) := \ker d_k \, / \, \overline{\mathrm{Im}(d_{k-1})},
\]
where the closure is taken in the $L^2$-norm induced by $g_{\operatorname{FS}}^{\mathcal{D}_{2m}}$. Since $g_{\operatorname{FS}}^{\mathcal{D}_{2m}}$ is complete K\"ahler metric, every $L^2$-cohomology group admits a unique harmonic representative. Furthermore, we have
\[
H_2^k(\mathcal{D}_{2m}) \cong \mathcal{H}_2^k(\mathcal{D}_{2m}),
\]
where $\mathcal{H}_2^k(\mathcal{D}_{2m})$ denotes the space of square‑integrable harmonic $k$-forms. Therefore, for each bidegree $(p,q)$, we have the following isomorphism
\[
H_2^{p,q}(\mathcal{D}_{2m}) \cong \mathcal{H}_2^{p,q}(\mathcal{D}_{2m}),
\]
and the space of harmonic forms decomposes as
\[
\mathcal{H}_2^k(\mathcal{D}_{2m}) = \bigoplus_{p+q=k} \mathcal{H}_2^{p,q}(\mathcal{D}_{2m}).
\]
We require the following vanishing result due to Donnelly \cite{d94} concerning 
$L^2$-cohomology away from the middle degree.
\begin{theorem}\label{donelly}
Let $(M,ds^2)$ be a complete K\"ahler manifold of complex dimension $n$ with K\"ahler form $\omega$. If there exists a $1$-form $\eta$ such that $\omega = d\eta$ with $\eta$ is bounded in the supremum norm, then for all $ k \neq n$, we have
\[
\mathcal{H}_2^k(M) = 0.
\]
\end{theorem}
To apply this result to $g_{\operatorname{FS}}^{\mathcal{D}_{2m}}$, we define the $1$-form $\eta = \Theta$ as follows:
\[
\Theta := -i \Big(\pa_z\log S_{\mathcal{D}_{2m}}(z)\,dz+\pa_w\log S_{\mathcal{D}_{2m}}(z)\,dw\Big),\qquad\omega = d\Theta,
\]
where 
\(
\bar\partial\partial=-\partial\bar\partial.
\)
To verify the hypothesis of Theorem~\ref{donelly}, we need to check that $\Theta$ is bounded in the supremum norm with respect to Fefferman-Szeg\H{o} metric $g_{\operatorname{FS}}^{\mathcal{D}_{2m}}$. This is equivalent to showing that the ratio
\begin{equation}\label{norm}
    \frac{|\Theta_z(X)|^2}{g_{\operatorname{FS},z}^{\mathcal{D}_{2m}}(X,X)}
\end{equation}
remains bounded for all $z\in\mathcal{D}_{2m}$ and non-zero tangent vectors $X\in T_z\mathcal{D}_{2m}$.


By the transformation rule for the Fefferman-Szeg\H{o} kernel (Proposition \ref{FS-transformation}), for any $T\in \operatorname{Aut}(\mathcal{D}_{2m})$, which satisfies Definition \ref{biholo_cond}, the $1$-form $\Theta$ transforms as:
\begin{equation*}
    (T^*\Theta)(z)=\Theta(z)+\frac{2}{3}\,\partial\log \det J_{\mathbb{C}}T(z).
\end{equation*}
Moreover, the Fefferman-Szeg\H{o} metric $g_{\operatorname{FS}}^{\mathcal{D}_{2m}}$ is invariant under $T$, i.e., 
$g_{\operatorname{FS}}^{\mathcal{D}_{2m}}=T^*g_{\operatorname{FS}}^{\mathcal{D}_{2m}}.$
Hence, we obtain the following transformation rule for (\ref{norm}):
\begin{equation}\label{pull1}
\frac{|\Theta_{T(z)}(T_*X)|^2}{g_{\operatorname{FS},T(z)}^{\mathcal{D}_{2m}}(T_*X,T_*X)}=\frac{|\Theta_z(X)+\frac{2}{3}\,\partial\log \det J_{\mathbb{C}}T(z)|^2}{g_{\operatorname{FS},z}^{\mathcal{D}_{2m}}(X,X)}.
\end{equation}

For any point \((z_0,w_0)\in\mathcal{D}_{2m}\), let $T$ be the inverse of the automorphism $F$ defined in \eqref{auto_eggdoms}, such that $T(0,w)=(z_{0},w_{0})$ with $w=\abs{w_{0}}(1-\abs{z_{0}}^{2})^{-1/(2m)}$. Hence, (\ref{pull1}) implies the following estimate:
\begin{equation}\label{esi}
    \frac{|\Theta_{(z_0,w_0)}(T_*X)|^2}{g_{\operatorname{FS},(z_0,w_0)}^{\mathcal{D}_{2m}}(T_*X,T_*X)}\le 2\left(
\frac{|\Theta_{(0,w)}(X)|^2}{g_{\operatorname{FS},(0,w)}^{\mathcal{D}_{2m}}(X,X)}+\left(\frac{2}{3}\right)^2\frac{|\partial\log\det J_{\mathbb{C}}T(0,w)|^2}{g_{\operatorname{FS},(0,w)}^{\mathcal{D}_{2m}}(X,X)}\right).
\end{equation}

Now, we prove that the left-hand side of \eqref{estimate} is bounded. The boundedness of the first term on the right-hand side of \eqref{esi} is follows from the fact that \((0,w)\) lies in the slice \(\{z=0\}\). Indeed, 
from \eqref{Psi-av}, we have
\[\partial_z\log S_{\mathcal{D}_{2m}}|_{z=0}=0\qquad\partial_w\log S_{\mathcal{D}_{2m}}|_{z=0}=\overline{w}~\Psi_v\] $\Psi_v=\partial_v\Psi|_{a=0}$ and $a=|z|^2,~v=|w|^2$. So, the $1$-form $\Theta=-i\partial\log S_{\mathcal{D}_{2m}}$ at the slice $\{z=0\}$ satisfies
\[\Theta_{(0,w)}=-i\overline{w}\Psi_v dw,\] and hence
\[\frac{|\Theta_{(0,w)}(X)|^2}{g_{\operatorname{FS},(0,w)}^{\mathcal{D}_{2m}}(X,X)}\leq\frac{|\overline{w}\Psi_v|^2|X_2|^2}{g_{2\overline{2}}(0,w)|X_2|^2}=\frac{|\overline{w}\Psi_v|^2}{g_{2\overline{2}}(0,w)}.\]
From straightforward computations, we know
\[\Psi_v=-\frac{m-1}{(m+2)-(m-1)v}+\frac{2}{1-v},\] and recall that \[g_{2\overline{2}}(0,w)=\frac{(2-rt^2)(1-rt)^2}{t^2(1-r)^2}\] where $r=(m-1)/(m+2)$ and $t=(1-v)/(1-rv)$. Using $1-rt=(1-r)/(1-rv)$ we get $(1-rt)/t=(1-r)/(1-v)$, this implies $\text{ as }v\to 1^-$,
\begin{align*}
    g_{2\overline{2}}(0,w)=\frac{(1-r)^2}{(1-v)^2}\cdot\frac{2-rt^2}{(1-r)^2}=\frac{2-rt^2}{(1-v)^2}&\sim\frac{2}{(1-v)^2},\\
    |\Theta_{(0,w)}(X)|^2=v|\Psi_v|^2&\sim\frac{4}{(1-v)^2},
\end{align*}
where $f(v)\sim g(v)$ as $v\to 1^-$ means $f/g\to 1$ as $v\to 1^-$. This implies, \[\lim_{v \to 1^-} \frac{|\overline{w}\Psi_v|^2}{g_{2\overline{2}}(0,w)}=  2.\]
Consequently, the first term on the right-hand side of \eqref{esi} is bounded near the boundary. Since this term is smooth in the interior of $\mathcal{D}_{2m}$, we conclude that it is uniformly bounded on the entire domain $\mathcal{D}_{2m}$. 

For the second term, we have the following estimate:
\begin{lemma}\label{estimate}
    For the automorphism $T$ mapping $(0,w)$ to $(z_0,w_0)$, the term $\partial\log\det J_\mathbb{C}T$
    is uniformly bounded in the Fefferman-Szeg\H{o} metric norm.
\end{lemma}
\begin{proof}
For the automorphism \(T\) constructed above, a straightforward calculation reveals that the logarithmic derivative of the complex Jacobian determinant at the points $(0,w)\in\mathcal{D}_{2m}$ is given by
\[
\partial\log\det J_{\mathbb{C}}T(0,w)= -\bar{z}_{0}\Bigl(2+\frac{1}{m}\Bigr)\,dz.
\]
By the explicit expression for the metric component $g_{1\bar1}(0,w)$ in Lemma \ref{FSmetric-Ricci}, we have:
\[
\abs{dz}_{g_{FS}^{\mathcal{D}_{2m}}}=\sqrt{g^{1\bar1}(0,w)}\le C<\infty.
\]
Since $\abs{z_{0}}<1$, it follows that
\[
|\partial\log\det J_{\mathbb{C}}T(0,w)|_{g_{\operatorname{FS}}^{\mathcal{D}_{2m}}} = |\bar{z}_{0}|\Bigl(2+\frac{1}{m}\Bigr)\,|dz|_{g_{\operatorname{FS}}^{\mathcal{D}_{2m}}} \le C \Bigl(2+\frac{1}{m}\Bigr)<\infty.
\]
The proof is completed. 
\end{proof}
Combining \eqref{esi} with Lemma \ref{estimate}, we concluded that \eqref{norm} is bounded on $\mathcal{D}_{2m}$. By Theorem \ref{donelly}, the proof of Theorem \ref{van-L2-coho} is completed.

\subsection*{Acknowledgements}
The authors sincerely thank Professor Xiaoshan Li for bringing the relevant reference \cite{y25} to their attention and for his constant encouragement and invaluable support throughout the preparation of this work. They also gratefully acknowledge the financial support provided by Wuhan University. The first author was partially supported by NSFC Grant No. 12361131577, and the second author by NSFC Grant No. 12271411, both via the grant held by Professor Xiaoshan Li.


\end{document}